\documentclass[11pt]{article}

\usepackage{graphicx, subfigure}
\usepackage[top=.75in,bottom=.75in,left=.75in,right=.75in]{geometry}

\usepackage{cite}

\let\oldbibliography\thebibliography
\renewcommand{\thebibliography}[1]{%
  \oldbibliography{#1}%
  \setlength{\itemsep}{0pt}%
}

\usepackage{amsfonts, amsmath, amssymb, amsthm, constants,comment}
\usepackage[usenames]{color}
\usepackage{dsfont}

\definecolor{darkred}{RGB}{100,0,0}
\definecolor{darkgreen}{RGB}{0,100,0}
\definecolor{darkblue}{RGB}{0,0,150}

\usepackage{hyperref}
\hypersetup{colorlinks=true, linkcolor=darkred, citecolor=darkgreen, urlcolor=darkblue}
\usepackage{url}

\newtheorem{theorem}{Theorem}

\newtheorem{lemma}{Lemma}

\theoremstyle{remark}
\newtheorem{remark}{Remark}

\theoremstyle{definition}
\newtheorem{example}{Example}

\DeclareMathOperator*{\E}{\mathbb{E}}

\DeclareMathOperator*{\rank}{rank}

\DeclareMathOperator*{\vecto}{vec}

\DeclareMathOperator*{\argmax}{arg\, max}
\DeclareMathOperator*{\argmin}{arg\, min}

\def \R {\mathbb{R}}

\def \Z {\mathbb{Z}}
\def \P {\mathbb{P}}

\def \I {\mathbb{I}}

\def \eps {\varepsilon}

\def \W {\Omega}
\def \moo| {\langle}
\def \< {\langle }
\def \> {\rangle }
\def \^ {\widehat}

\newcommand{\norm}[1]{\left \|#1\right \|}

\newcommand{\twonorm}[1]{\norm{#1}_2}
\newcommand{\inftynorm}[1]{\norm{#1}_\infty}
\newcommand{\opnorm}[1]{\norm{#1}}
\newcommand{\fronorm}[1]{\norm{#1}_F}
\newcommand{\nucnorm}[1]{\norm{#1}_*}

\newcommand{\abs}[1]{\left | #1 \right |}

\def\minim{\mathop{\hbox{minimize}}}

\def\minimize#1{\displaystyle\minim_{#1}}

\def\st{\mathop{\hbox{subject to}}}
\def\half{{\textstyle{\frac{1}{2}}}}

\renewcommand{\Pr}[1]{\P \left( #1 \rule{0mm}{3mm}\right)}

\newcommand{\vect}[1]{\boldsymbol{#1}}
\newcommand{\mat}[1]{\boldsymbol{#1}}

\def \u {\vect{u}}
\def \v {\vect{v}}

\def \mE {\mat{E}}

\def \mP {\mat{P}}
\def \mQ {\mat{Q}}

\def \mX {\mat{X}}

\def \A {\mat{A}}
\def \B {\mat{B}}

\def \I {\mat{I}}

\def \M {\mat{M}}

\def \U {\mat{U}}
\def \V {\mat{V}}
\def \W {\mat{W}}
\def \X {\mat{X}}
\def \Y {\mat{Y}}
\def \Z {\mat{Z}}

\def \bzero {\vect{0}}
\def \bone {\vect{1}}
\def \bSigma {\mat{\Sigma}}

\def \Mhat {\widehat{\M}}
\def \Mhatel {\widehat{M}}
\def \meas {n}

\def \loglike {\mathcal{L}}

\newcommand{\Exp}[1]{\mathbb{E}\left[ #1\right]}

\newcommand{\ind}[1]{\mathds{1}_{[#1]}}
\newcommand{\infind}[1]{\mathds{I}_{[#1]}}
\newcommand{\ip}[2]{\ensuremath{\left\langle #1,#2\right\rangle}}
\newcommand{\had}{\circ}

\definecolor{noteColor}{RGB}{200,50,50}

\title{1-Bit Matrix Completion}

\author{
Mark A. Davenport, Yaniv Plan, Ewout van den Berg, Mary Wootters\thanks{ M.A.\ Davenport  is with the School of Electrical and Computer Engineering, Georgia Institute of Technology, Atlanta, GA.  Email: \{\href{mailto:mdav@gatech.edu}{mdav@gatech.edu}\}. \protect \\
\indent Y.\ Plan is with the Department of Mathematics, University of Michigan, Ann Arbor, MI.  Email: \{\href{mailto:yplan@umich.edu}{yplan@umich.edu}\}. \protect \\
\indent E.\ van den Berg is with the IBM T.J.\ Watson Research Center, Yorktown Heights, NY. Email:
\{\href{mailto:evandenberg@us.ibm.com}{evandenberg@us.ibm.com}\}. \protect \\
\indent M.\ Wootters is with the Department of Mathematics, University of Michigan, Ann Arbor, MI. Email:
\{\href{mailto:wootters@umich.edu}{wootters@umich.edu}\}. \protect \\
\indent This work was partially supported by NSF grants DMS-0906812, DMS-1004718,
DMS-1103909,  CCF-0743372, and CCF-1350616 and NRL grant N00173-14-2-C001.}}

\date{September 2012 (Revised May 2014)}

\begin{document}

\maketitle

\begin{abstract}
In this paper we develop a theory of matrix completion for the extreme case of noisy 1-bit observations.  Instead of observing a subset of the real-valued entries of a matrix $\M$, we obtain a small number of binary (1-bit) measurements generated according to a probability distribution determined by the real-valued entries of $\M$.  The central question we ask is whether or not it is possible to obtain an accurate estimate of $\M$ from this data.  In general this would seem impossible, but we show that the maximum likelihood estimate under a suitable constraint returns an accurate estimate of $\M$ when $\|\M\|_{\infty} \le \alpha$ and $\rank(\M) \leq r$.
If the log-likelihood is a concave function (e.g., the logistic or probit observation models), then we can obtain this maximum likelihood estimate by optimizing a convex program.  In addition, we also show that if instead of recovering $\M$ we simply wish to obtain an estimate of the distribution generating the 1-bit measurements, then we can eliminate the requirement that $\|\M\|_{\infty} \le \alpha$.  For both cases, we provide lower bounds showing that these estimates are near-optimal.  We conclude with a suite of experiments that both verify the implications of our theorems as well as illustrate some of the practical applications of 1-bit matrix completion.  In particular, we compare our program to standard matrix completion methods on movie rating data in which users submit ratings from 1 to 5.  In order to use our program, we quantize this data to a single bit, but we allow the standard matrix completion program to have access to the original ratings (from 1 to 5).  Surprisingly, the approach based on binary data performs significantly better.
\end{abstract}

\section{Introduction}

The problem of recovering a matrix from an incomplete sampling of its entries---also known as {\em matrix completion}---arises in a wide variety of practical situations.  In many of these settings, however, the observations are not only incomplete, but also highly {\em quantized}, often even to a single bit.  In this paper we consider a statistical model for such data where instead of observing a real-valued entry as in the original matrix completion problem, we are now only able to see a positive or negative rating.  This binary output is generated according to a probability distribution which is parameterized by the corresponding entry of the unknown low-rank matrix $\M$. The central question we ask in this paper is: ``Given observations of this form, can we recover the underlying matrix?''

Questions of this form are often asked in the context of \textit{binary PCA} or \textit{logistic PCA}.  There are a number of compelling algorithmic papers on these subjects, including \cite{de2006principal, Schein03, Khan10nips, collins2001generalization, Tipping98}, which suggest positive answers on simulated and real-world data.  In this paper, we give the first theoretical accuracy guarantees under a \textit{generalized linear model}.  We show that $O(r d)$ binary observations are sufficient to accurately recover a $d \times d$, rank-$r$ matrix by convex programming.  Our theory is inspired by the unquantized matrix completion problem and the closely related problem of 1-bit compressed sensing, described below.


\subsection{Matrix completion}

Matrix completion arises in a wide variety of practical contexts, including collaborative filtering~\cite{goldberg1992using}, system identification~\cite{liu2009interior}, sensor localization~\cite{biswas2006semidefinite,singer2009remark,singer2010uniqueness}, rank aggregation~\cite{gleich2011rank}, and many more. While many of these applications have a relatively long history, recent advances in the closely related field of compressed sensing~\cite{donoho2006compressed,candes2006,Davenport2012} have enabled a burst of progress in the last few years, and we now have a strong base of theoretical results concerning matrix completion~\cite{gross2011recovering, candes2009exact, candes2010power, keshavan2010matrix, keshavan2010noisy, negahban2010restricted, koltchinskii2011nuclear, candes2010matrix, recht2011simpler, rohde2011estimation, klopp2011rank, gaiffas2010sharp, klopp2011high, koltchinskii2012neumann}.  A typical result from this literature is that a generic $d \times d$ matrix of rank $r$ can be exactly recovered from $O(r \, d \, \mathrm{polylog}(d))$ randomly chosen entries.  Similar results can be established in the case of noisy observations and approximately low-rank matrices~\cite{keshavan2010noisy, negahban2010restricted, koltchinskii2011nuclear, candes2010matrix, rohde2011estimation, klopp2011rank, gaiffas2010sharp, klopp2011high, koltchinskii2012neumann}.

Although these results are quite impressive, there is an important gap between the statement of the problem as considered in the matrix completion literature and many of the most common applications discussed therein.  As an example, consider collaborative filtering and the now-famous ``Netflix problem.''  In this setting, we assume that there is some unknown matrix whose entries each represent a rating for a particular user on a particular movie.  Since any user will rate only a small subset of possible movies, we are only able to observe a small fraction of the total entries in the matrix, and our goal is to infer the unseen ratings from the observed ones.  If the rating matrix has low rank, then this would seem to be the exact problem studied in the matrix completion literature.  However, there is a subtle difference: the theory developed in this literature generally assumes that observations consist of (possibly noisy) continuous-valued entries of the matrix, whereas in the Netflix problem the observations are ``quantized'' to the set of integers between 1 and 5.  If we believe that it is possible for a user's true rating for a particular movie to be, for example, 4.5, then we must account for the impact of this ``quantization noise'' on our recovery.  Of course, one could potentially treat quantization simply as a form of bounded noise, but this is somewhat unsatisfying because the ratings aren't just quantized --- there are also hard limits placed on the minimum and maximum allowable ratings.  (Why should we suppose that a movie given a rating of 5 could not have a true underlying rating of 6 or 7 or 10?)  The inadequacy of standard matrix completion techniques in dealing with this effect is particularly pronounced when we consider recommender systems where each rating consists of a single bit representing a positive or negative rating (consider for example rating music on Pandora, the relevance of advertisements on Hulu, or posts on sites such as MathOverflow). In such a case, the assumptions made in the existing theory of matrix completion do not apply, standard algorithms are ill-posed, and alternative theory is required.

\subsection{Statistical learning theory and matrix completion}

While the theory of matrix completion gained quite a bit of momentum following advances in compressed sensing, earlier results from Srebro et al.~\cite{srebro2004generalization, srebro2004maximum} also addressed this problem from a slightly different perspective rooted in the framework of  statistical learning theory.  These results also deal with the binary setting that we consider here.  They take a model-free approach and prove \textit{generalization error bounds}; that is, they give conditions under which good agreement on the observed data implies good agreement on the entire matrix.  For example, in~\cite{srebro2004maximum}, agreement is roughly measured via the fraction of predicted signs that are correct, but this can also be extended to other notions of agreement \cite{srebro2004generalization}.  From the perspective of statistical learning theory, this corresponds to bounding the generalization error under various classes of loss functions.

One important difference between the statistical learning approach and the path taken in this paper is that here we focus on {\em parameter estimation} --- that is, we seek to recover the matrix $\M$ itself (or the distribution parameterized by $\M$ that governs our observations) --- although we also prove generalization error bounds en route to our main results on parameter and distribution recovery. We discuss the relationship between our approach and the statistical learning approach in more detail in Section~\ref{ssec:recover_distribution} (see Remark \ref{remark:generror}).  Briefly, our generalization error bounds correspond to the case where the loss function is the {\em  log likelihood}, and do not seem to fit directly within the framework of the existing literature.

\subsection{1-bit compressed sensing and sparse logistic regression}

As noted above, matrix completion is closely related to the field of compressed sensing, where a theory to deal with single-bit quantization has recently been developed~\cite{Boufounos2008, Jacques2011, pv-1-bit, plan2012robust, Gupta2010, laska2011regime}.  In compressed sensing, one can recover an $s$-sparse vector in $\mathbb{R}^d$ from $O( s \log( d/s ))$ random linear measurements---several different random measurement structures are compatible with this theory.  In 1-bit compressed sensing, only the signs of these measurements are observed, but an $s$-sparse signal can still be approximately recovered from the same number of measurements~\cite{Jacques2011, pv-1-bit, plan2012robust, ai2012one}.  However, the only measurement ensembles which are currently known to give such guarantees are Gaussian or sub-Gaussian \cite{ai2012one}, and are thus of a quite different flavor than the kinds of samples obtained in the matrix completion setting. A similar theory is available for the closely related problem of sparse binomial regression, which considers more classical statistical models~\cite{bach2010self, bunea2008honest, plan2012robust, kakade2009learning, meier2008group, negahban2010unified, ravikumar2010high, van2008high} and allows non-Gaussian measurements.  Our aim here is to develop results for matrix completion of the same flavor as 1-bit compressed sensing and sparse logistic regression.

\subsection{Challenges}
\label{sec:challenges}

In this paper, we extend the theory of matrix completion to the case of 1-bit observations.  We consider a general observation model but focus mainly on two particular possibilities: the models of logistic and probit regression. We discuss these models in greater detail in Section~\ref{sec:model}, but first we note that several new challenges arise when trying to leverage results in 1-bit compressed sensing and sparse logistic regression to develop a theory for 1-bit matrix completion. First, matrix completion is in some sense a more challenging problem than compressed sensing.  Specifically, some additional difficulty arises because the set of low-rank matrices is ``coherent" with single entry measurements (see~\cite{gross2011recovering}).  In particular, the sampling operator does not act as a near-isometry on all matrices of interest, and thus the natural analogue to the restricted isometry property from compressed sensing cannot hold in general---there will always be certain low-rank matrices that we cannot hope to recover without essentially sampling every entry of the matrix.  For example, consider a matrix that consists of a single nonzero entry (which we might never observe).  The typical way to deal with this possibility is to consider a reduced set of low-rank matrices by placing restrictions on the entry-wise maximum of the matrix or its singular vectors---informally, we require that the matrix is not too ``spiky".

We introduce an entirely new dimension of ill-posedness by restricting ourselves to 1-bit observations. To illustrate this, we describe one version of 1-bit matrix completion in more detail (the general problem definition is given in Section \ref{sec:model} below).  Consider a $d \times d$ matrix $\M$ with rank $r$. Suppose we observe a subset $\Omega$ of entries of a matrix $\Y$.  The entries of $\Y$ depend on $\M$ in the following way:
\begin{equation}
\label{eq:noisy 1 bit}
    Y_{i,j} = \begin{cases} +1 & \mathrm{if} \ M_{i,j} + Z_{i,j} \geq 0 \\ -1 & \mathrm{if} \ M_{i,j} + Z_{i,j} < 0 \end{cases}
\end{equation}
where $\Z$ is a matrix containing noise. This latent variable model is the direct analogue to the usual 1-bit compressed sensing observation model.  In this setting, we view the matrix $\M$ as more than just a parameter of the distribution of $\Y$; $\M$ represents the real underlying quantity of interest that we would like to estimate. Unfortunately, in what would seem to be the most benign setting---when $\Omega$ is the set of all entries, $\Z = \bzero$, and $\M$ has rank 1 and a bounded entry-wise maximum---the problem of recovering $\M$ is ill-posed.  To see this, let $\M = \u\v^*$ for any vectors $\u,\v \in \R^d$, and for simplicity assume that there are no zero entries in $\u$ or $\v$.  Now let $\widetilde{\u}$ and $\widetilde{\v}$ be any vectors with the same sign pattern as $\u$ and $\v$ respectively.  It is apparent that both $\M$ and $\widetilde{\M} = \widetilde{\u}\widetilde{\v}^*$ will yield the same observations $\Y$, and thus $\M$ and $\widetilde{\M}$ are indistinguishable.  Note that while it is obvious that this 1-bit measurement process will destroy any information we have regarding the scaling of $\M$, this ill-posedness remains even if we knew something about the scaling a priori (such as the Frobenius norm of $\M$).  For any given set of observations, there will always be radically different possible matrices that are all consistent with observed measurements.

After considering this example, the problem might seem hopeless.  However, an interesting surprise is that when we add noise to the problem (that is, when $\Z \neq \bzero$ is an appropriate stochastic matrix) the picture completely changes---this noise has a ``dithering'' effect and the problem becomes well-posed.  In fact, we will show that in this setting we can sometimes recover $\M$ to the same degree of accuracy that is possible when given access to completely unquantized measurements!  In particular, under appropriate conditions, $O(rd)$ measurements are sufficient to accurately recover $\M$.

\subsection{Applications}

The problem of 1-bit matrix completion arises in nearly every application that has been proposed for ``unquantized'' matrix completion.  To name a few:
\begin{itemize}
\item {\bf Recommender systems:}  As mentioned above, collaborative filtering systems often involve discretized recommendations~\cite{goldberg1992using}.  In many cases, each observation will consist simply of a ``thumbs up'' or ``thumbs down" thus delivering only 1 bit of information (consider for example rating music on Pandora, the relevance of advertisements on Hulu, or posts on sites such as MathOverflow).  Such cases are a natural application for 1-bit matrix completion.
\item {\bf Analysis of survey data:} Another potential application for matrix completion is to analyze incomplete survey data.  Such data is almost always heavily quantized since people are generally not able to distinguish between more than $7 \pm 2$ categories~\cite{millerMagic}. 1-bit matrix completion provides a method for analyzing incomplete (or potentially even complete) survey designs containing simple yes/no or agree/disagree questions.
\item {\bf Distance matrix recovery and multidimensional scaling:} Yet another common motivation for matrix completion is to localize nodes in a sensor network from the observation of just a few inter-node distances~\cite{biswas2006semidefinite,singer2009remark,singer2010uniqueness}.  This is essentially a special case of multidimensional scaling (MDS) from incomplete data~\cite{borgModern}.  In general, work in the area assumes real-valued measurements.  However, in the sensor network example (as well as many other MDS scenarios), the measurements may be very coarse and might only indicate whether the nodes are within or outside of some communication range.  While there is some existing work on MDS using binary data~\cite{greenMultivariate} and MDS using incomplete observations with other kinds of non-metric data~\cite{spenceSingle}, 1-bit matrix completion promises to provide a principled and unifying approach to such problems.
\item {\bf Quantum state tomography:} Low-rank matrix recovery from
  incomplete observations also has applications to quantum state
  tomography~\cite{gross2010quantum}.  In this scenario, mixed quantum
  states are represented as Hermitian matrices with nuclear norm equal
  to 1.  When the state is nearly pure, the matrix can be well
  approximated by a low-rank matrix and, in particular, fits the model
  given in Section~\ref{ssec:approx} up to a rescaling.  Furthermore,
  Pauli-operator-based measurements give probabilistic binary outputs.
  However, these are based on the inner products with the Pauli matrices, and thus of a slightly different flavor than the measurements considered in this paper.  Nevertheless, while we do not address this scenario directly, our theory of 1-bit matrix completion could easily be adapted to quantum state tomography.
\end{itemize}

\subsection{Notation}

We now provide a brief summary of some of the key notation used in this paper.  We use $[d]$ to denote the set of integers $\{1, \ldots, d\}$.  We use capital boldface to denote a matrix (e.g., $\M$) and standard text to denote its entries (e.g., $M_{i,j}$).  Similarly, we let $\bzero$ denote the matrix of all-zeros and $\bone$ the matrix of all-ones. We let $\|\M\|$ denote the operator norm of $\M$, $\fronorm{\M} = \sqrt{\sum_{i,j} M_{i,j}^2}$ denote the Frobenius norm of $\M$, $\nucnorm{\M}$ denote the nuclear or Schatten-1 norm of $\M$ (the sum of the singular values), and  $\inftynorm{\M} = \max_{i,j} \abs{M_{i,j}}$ denote the entry-wise infinity-norm of $\M$. We will use the Hellinger distance, which, for two scalars $p,q \in [0,1]$, is given by
\[
    d_H^2(p,q) := (\sqrt{p} - \sqrt{q})^2 + (\sqrt{1-p} - \sqrt{1-q})^2.
\]
This gives a standard notion of distance between two binary probability distributions.  We also allow the Hellinger distance to act on matrices via the average Hellinger distance over their entries:  for matrices $\mP, \mQ \in [0,1]^{d_1 \times d_2}$, we define
\[
    d^2_H(\mP,\mQ) = \frac{1}{d_1 d_2} \sum_{i,j} d^2_H(P_{i,j}, Q_{i,j}).
\]
Finally, for an event $\mathcal{E}$ ,$\ind{\mathcal{E}}$ is the indicator function for that event, i.e., $\ind{\mathcal{E}}$ is $1$ if $\mathcal{E}$ occurs and $0$ otherwise.

\subsection{Organization of the paper}

We proceed in Section~\ref{sec:problem} by describing the 1-bit matrix completion problem in greater detail.  In Section~\ref{sec:main} we state our main results.  Specifically, we propose a pair of convex programs for the 1-bit matrix completion problem and establish upper bounds on the accuracy with which these can recover the matrix $\M$ and the distribution of the observations $\Y$.  We also establish lower bounds, showing that our upper bounds are nearly optimal.   In Section~\ref{sec:sims} we describe numerical implementations of our proposed convex programs and demonstrate their performance on a number of synthetic and real-world examples.  Section~\ref{sec:disc} concludes with a brief discussion of future directions. The proofs of our main results are provided in the appendix.

\section{The 1-bit matrix completion problem}
\label{sec:problem}

\subsection{Observation model}
\label{sec:model}

We now introduce the more general observation model that we study in this paper. Given a matrix $\M \in \R^{d_1\times d_2}$, a subset of indices $\Omega \subset [d_1] \times [d_2]$, and a differentiable function $f: \R \rightarrow [0,1]$, we observe
\begin{equation}
\label{eq:observations}
    Y_{i,j} = \begin{cases} +1 & \mathrm{with~probability} \ f(M_{i,j}), \\
        -1 & \mathrm{with~probability} \ 1- f(M_{i,j}) \end{cases}
    \quad \text{ for $(i,j) \in \Omega$.}
\end{equation}
We will leave $f$ general for now and discuss a few common choices just below.  As has been important in previous work on matrix completion, we assume that $\Omega$ is chosen at random with $\E |\Omega| = \meas $.  Specifically, we assume that $\Omega$ follows a binomial model in which each entry $(i,j) \in [d_1] \times [d_2]$ is included in $\Omega$ with probability $\frac{\meas }{d_1 d_2}$, independently.

Before discussing some particular choices for $f$, we first note that while the observation model described in~\eqref{eq:observations} may appear on the surface to be somewhat different from the setup in~\eqref{eq:noisy 1 bit}, they are actually equivalent if $f$ behaves likes a cumulative distribution function.  Specifically, for the model in~\eqref{eq:noisy 1 bit}, if $\Z$ has i.i.d.\ entries, then by setting $f(x) := \Pr{Z_{1,1} \geq - x}$, the model in~\eqref{eq:noisy 1 bit} reduces to that in~\eqref{eq:observations}.  Similarly, for any choice of $f(x)$ in~\eqref{eq:observations}, if we define $\Z$ as having i.i.d.\ entries drawn from a distribution whose cumulative distribution function is given by $F_Z(x) = \Pr{z \le x} = 1 - f(-x)$, then~\eqref{eq:observations} reduces to~\eqref{eq:noisy 1 bit}.  Of course, in any given situation one of these observation models may seem more or less natural than the other---for example,~\eqref{eq:noisy 1 bit} may seem more appropriate when $\M$ is viewed as a latent variable which we might be interested in estimating, while~\eqref{eq:observations} may make more sense when $\M$ is viewed as just a parameter of a distribution.  Ultimately, however, the two models are equivalent.

We now consider two natural choices for $f$ (or equivalently, for $\Z$):
\begin{example}[Logistic regression/Logistic noise]
\label{ex:logisticRegression}
The logistic regression model, which is common in statistics, is captured by~\eqref{eq:observations} with $f(x) = \frac{e^x}{1 + e^x}$ and by~\eqref{eq:noisy 1 bit} with $Z_{i,j}$ i.i.d.\ according to the standard logistic distribution.
\end{example}

\begin{example}[Probit regression/Gaussian noise]
\label{ex:latentVariableSelection}
The probit regression model is captured by~\eqref{eq:observations} by setting $f(x) = 1- \Phi(-x/\sigma) = \Phi(x/\sigma)$ where $\Phi$ is the cumulative distribution function of a standard Gaussian and by~\eqref{eq:noisy 1 bit} with $Z_{i,j}$ i.i.d.\ according to a mean-zero Gaussian distribution with variance $\sigma^2$.
\end{example}

\subsection{Approximately low-rank matrices}
\label{ssec:approx}

The majority of the literature on matrix completion assumes that the first $r$ singular values of $\M$ are nonzero and the remainder are exactly zero.  However, in many applications the singular values instead exhibit only a gradual decay towards zero.  Thus, in this paper we allow a relaxation of the assumption that $\M$ has rank exactly $r$.  Instead, we assume that $\nucnorm{\M} \leq \alpha \sqrt{r d_1 d_2}$, where $\alpha$ is a parameter left to be determined, but which will often be of constant order.  In other words, the singular values of $\M$ belong to a scaled $\ell_1$ ball.  In compressed sensing, belonging to an $\ell_p$ ball with $p \in (0, 1]$ is a common relaxation of exact sparsity; in matrix completion, the nuclear-norm ball (or Schatten-1 ball) plays an analogous role.  See \cite{chatterjee2012matrix} for further compelling reasons to use the nuclear-norm ball relaxation.

The particular choice of scaling, $\alpha \sqrt{r d_1 d_2}$, arises from the following considerations.  Suppose that each entry of $\M$ is bounded in magnitude by $\alpha$ and that $\rank(\M) \leq r$.  Then
\[
    \nucnorm{\M} \leq \sqrt{r} \fronorm{\M} \leq \sqrt{r d_1 d_2} \inftynorm{\M} \leq \alpha \sqrt{r d_1 d_2}.
\]
Thus, the assumption that $\nucnorm{\M} \leq \alpha \sqrt{r d_1 d_2}$ is a relaxation of the conditions that $\rank(\M) \le r$ and $\inftynorm{\M} \le \alpha$.  The condition that $\inftynorm{\M} \le \alpha$ essentially means that the probability of seeing a $+1$ or $-1$ does not depend on the dimension.  It is also a way of enforcing that $\M$ should not be too ``spiky''; as discussed above this is an important assumption in order to make the recovery of $\M$ well-posed (e.g., see~\cite{negahban2010restricted}).

\section{Main results}
\label{sec:main}

We now state our main results.  We will have two goals---the first is to accurately recover $\M$ itself, and the second is to accurately recover the distribution of $\Y$ given by $f(\M)$.\footnote{Strictly speaking, $f(\M) \in [0,1]^{d_1\times d_2}$ is simply a matrix of scalars, but these scalars implicitly define the distribution of $\Y$, so we will sometimes abuse notation slightly and refer to $f(\M)$ as the distribution of $\Y$.}  All proofs are contained in the appendix.

\subsection{Convex programming}

In order to approximate either $\M$ or $f(\M)$, we will maximize the log-likelihood function of the optimization variable $\X$ given our observations subject to a set of convex constraints.  In our case, the log-likelihood function is given by
\[
    \loglike_{\Omega, \Y}(\X) := \sum_{(i,j) \in \Omega}  \left(\ind{Y_{i,j} = 1} \log(f(X_{i,j})) + \ind{Y_{i,j} = -1} \log(1 - f(X_{i,j})) \right).
\]
To recover $\M$, we will use the solution to the following program:
\begin{equation}
\label{eq:max likelihood with inf norm}
    \Mhat = \argmax_{\X} \loglike_{\Omega, \Y}(\X)  \qquad \mathrm{subject~to} \qquad \nucnorm{\mX} \leq \alpha \sqrt{r d_1 d_2} \quad \mathrm{and} \quad \inftynorm{\mX} \leq \alpha.
\end{equation}
To recover the distribution $f(\M)$, we need not enforce the infinity-norm constraint, and will use the following simpler program:
\begin{equation}
\label{eq:max likelihood}
    \Mhat = \argmax_{\X}  \loglike_{\Omega, \Y}(\X)  \qquad \mathrm{subject~to} \qquad \nucnorm{\X} \leq \alpha \sqrt{r d_1 d_2}
\end{equation}

In many cases, $\loglike_{\Omega, \Y}(\X)$ is a concave function and thus the above programs are convex.  This can be easily checked in the case of the logistic model and can also be verified in the case of the probit model (e.g., see~\cite{zymnis2010compressed}).

\subsection{Recovery of the matrix}
\label{ssec:recovery of the matrix}

We now state our main result concerning the recovery of the matrix $\M$.  As discussed in Section~\ref{sec:challenges} we place a ``non-spikiness'' condition on $\M$ to make recovery possible; we enforce this with an infinity-norm constraint. Further, some assumptions must be made on $f$ for recovery of $\M$ to be feasible.  We define two quantities $L_\alpha$ and $\beta_\alpha$ which control the ``steepness" and ``flatness" of $f$, respectively:
\begin{equation}
\label{eq:lipschitz}
    L_\alpha := \sup_{|x| \leq \alpha} \frac{\abs{f'(x)}}{f(x) (1 - f(x))} \qquad \mathrm{and} \qquad \beta_\alpha := \sup_{\abs{x} \leq \alpha} \frac{f(x) (1 - f(x)) }{(f'(x))^2}.
\end{equation}
In this paper we will restrict our attention to $f$ such that $L_\alpha$ and $\beta_\alpha$ are well-defined.  In particular, we assume that $f$ and $f'$ are non-zero in $[-\alpha, \alpha]$.  This assumption is fairly mild---for example, it includes the logistic and probit models (as we will see below in Remark~\ref{specific models}).  The quantity $L_\alpha$ appears only in our upper bounds, but it is generally well behaved.  The quantity $\beta_\alpha$ appears both in our upper and lower bounds. Intuitively, it controls the ``flatness" of $f$ in the interval $[-\alpha,\alpha]$---the flatter $f$ is, the larger $\beta_\alpha$ is.  It is clear that some dependence on $\beta_\alpha$ is necessary.  Indeed, if $f$ is perfectly flat, then the magnitudes of the entries of $\M$ cannot be recovered, as seen in the noiseless case discussed in Section \ref{sec:challenges}. Of course, when $\alpha$ is a fixed constant and $f$ is a fixed function, both $L_\alpha$ and $\beta_\alpha$ are bounded by fixed constants independent of the dimension.

\begin{theorem}
\label{thm:recover M}
Assume that $\nucnorm{\M} \leq \alpha \sqrt{d_1 d_2 r}$ and $\|\M\|_\infty \le \alpha$.  Suppose that $\Omega$ is chosen at random following the binomial model of Section \ref{sec:model} with $\E|\Omega| = \meas $.  Suppose that $\Y$ is generated as in~\eqref{eq:observations}.  Let $L_\alpha$ and $\beta_\alpha$ be as in~\eqref{eq:lipschitz}. Consider the solution $\Mhat$ to~\eqref{eq:max likelihood with inf norm}.  Then with probability at least $1 - C_1/(d_1+d_2)$,
\[
    \frac{1}{d_1 d_2}\|\Mhat - \M\|_F^2 \leq C_\alpha  \sqrt{\frac{r(d_1 + d_2)}{\meas }}\sqrt{1 + \frac{(d_1 + d_2) \log(d_1 d_2)}{\meas }}
\]
with $C_\alpha := C_2 \alpha L_\alpha \beta_\alpha $. If $\meas  \ge (d_1+d_2)\log(d_1 d_2)$ then this simplifies to
\begin{equation}
\label{eq:tight bound on M}
    \frac{1}{d_1 d_2}\|\Mhat - \M\|_F^2 \leq \sqrt{2} C_\alpha  \sqrt{\frac{r(d_1 + d_2)}{\meas }}.
\end{equation}
Above, $C_1$ and $C_2$ are absolute constants.
\end{theorem}

Note that the theorem also holds if $\Omega = [d_1] \times [d_2]$, i.e., if we sample each entry exactly once or observe a complete realization of $\Y$.  Even in this context, the ability to accurately recover $\M$ is somewhat surprising.

\begin{remark}[Recovery in the logistic and probit models]\label{specific models}
The logistic model satisfies the hypotheses of Theorem~\ref{thm:recover M} with $\beta_\alpha =  \frac{(1 + e^\alpha)^2}{e^\alpha} \approx e^{\alpha}$ and $L_\alpha = 1$. The probit model has
\[
    \beta_\alpha \leq c_1 \sigma^2 e^{\frac{\alpha^2}{2 \sigma^2}} \quad \text{and} \quad L_\alpha \leq c_2 \frac{\frac{\alpha}{\sigma} + 1}{\sigma}
\]
where we can take $c_1 = \pi$ and $c_2 = 8$.  In particular, in the probit model the bound in~\eqref{eq:tight bound on M} reduces to
\begin{equation}
\label{eq:probit tight bound}
    \frac{1}{d_1 d_2}\|\Mhat - \M\|_F^2 \leq C \left(\frac{\alpha}{\sigma} + 1\right) \exp\left(\frac{\alpha^2}{2 \sigma^2}\right) \sigma \alpha \sqrt{\frac{r(d_1 + d_2)}{\meas }}.
\end{equation}
Hence, when $\sigma < \alpha$, increasing the size of the noise leads to significantly improved error bounds---this is not an artifact of the proof. We will see in Section \ref{ssec:optimality} that the exponential dependence on $\alpha$ in the logistic model (and on $\alpha/\sigma$ in the probit model) is intrinsic to the problem.  Intuitively, we should expect this since for such models, as $\| \M \|_\infty$ grows large, we can essentially revert to the noiseless setting where estimation of $\M$ is impossible.  Furthermore, in Section \ref{ssec:optimality} we will also see that when $\alpha$ (or $\alpha/\sigma$) is bounded by a constant, the error bound \eqref{eq:tight bound on M} is optimal up to a constant factor.  Fortunately, in many applications, one would expect $\alpha$ to be small, and in particular to have little, if any, dependence on the dimension.  This ensures that each measurement will always have a non-vanishing probability of returning $1$ as well as a non-vanishing probability of returning $-1$.
\end{remark}

\begin{remark}[Nuclear norm constraint]
The assumption made in Theorem~\ref{thm:recover M} (as well as Theorem~\ref{thm:recover distribution} below) that $\nucnorm{\M} \leq \alpha \sqrt{d_1 d_2 r}$ does {\em not} mean that we are requiring the matrix $\M$ to be low rank.  We express this constraint in terms of $r$ to aid the intuition of researchers well-versed in the existing literature on low-rank matrix recovery.  (If $\M$ is exactly rank $r$ and satisfies $\|\M\|_\infty \le \alpha$, then as discussed in Section~\ref{ssec:approx}, $\M$ will automatically satisfy this constraint.)  If one desires, one may simplify the presentation by replacing $\alpha \sqrt{d_1 d_2 r}$ with a parameter $\lambda$ and simply requiring $\nucnorm{\M} \leq \lambda$, in which case \eqref{eq:tight bound on M} reduces to
\[ d_H^2(f(\Mhat),f(\M)) \leq C_3 L_\alpha \beta_\alpha \frac{\lambda}{\sqrt{\min(d_1, d_2) \cdot \meas}}\]
for a numerical constant $C_3$.
\end{remark}

\subsection{Recovery of the distribution}
\label{ssec:distribution}

In many situations, we might not be interested in the underlying matrix $\M$, but rather in determining the distribution of the entries of $\Y$.  For example, in recommender systems, a natural question would be to determine the likelihood that a user would enjoy a particular unrated item.

Surprisingly, this distribution may be accurately recovered without any restriction on the infinity-norm of $\M$.  This may be unexpected to those familiar with the matrix completion literature in which ``non-spikiness'' constraints seem to be unavoidable. In fact, we will show in Section \ref{ssec:optimality} that the bound in Theorem~\ref{thm:recover distribution} is near-optimal; further, we will show that even under the added constraint that $\inftynorm{\M} \leq \alpha$, it would be impossible to estimate $f(\M)$ significantly more accurately.

\begin{theorem}
\label{thm:recover distribution}
Assume that $\nucnorm{\M} \leq \alpha \sqrt{d_1 d_2 r}$.  Suppose that $\Omega$ is chosen at random following the binomial model of Section \ref{sec:model} with $\E|\Omega| = \meas $.  Suppose that $\Y$ is generated as in~\eqref{eq:observations}, and let $L = \lim_{\alpha \rightarrow \infty} L_\alpha$. Let $\Mhat$ be the solution to~\eqref{eq:max likelihood}.  Then with probability at least $1 - C_1/(d_1+d_2)$,
\begin{equation}
\label{eq:distribution error bound}
    d_H^2(f(\Mhat), f(\M) ) \leq C_2 \alpha L \sqrt{\frac{r(d_1 + d_2)}{\meas }}\sqrt{1 + \frac{(d_1 + d_2) \log(d_1 d_2)}{\meas }}.
\end{equation}
Furthermore, as long as $\meas  \geq (d_1 + d_2) \log(d_1 d_2)$, we have
\begin{equation}
\label{eq:simple distribution error bound}
    d_H^2(f(\Mhat),f(\M)) \leq \sqrt{2} C_2 \alpha L \sqrt{\frac{r(d_1 + d_2)}{\meas }}.
\end{equation}
Above, $C_1$ and $C_2$ are absolute constants.
\end{theorem}

While $L=1$ for the logistic model, the astute reader will have noticed that for the probit model $L$ is unbounded---that is, $L_\alpha$ tends to $\infty$ as $\alpha \to \infty$.  $L$ would also be unbounded for the case where $f(x)$ takes values of $1$ or $0$ outside of some range (as would be the case in \eqref{eq:noisy 1 bit} if the distribution of the noise had compact support).   Fortunately, however, we can recover a result for these cases by enforcing an infinity-norm constraint, as described in Theorem \ref{thm:recover distribution with constraints} below.  Moreover, for a large class of functions, $f$, $L$ is indeed bounded.  For example, in the latent variable version of \eqref{eq:noisy 1 bit} if the entries $Z_{i,j}$ are at least as fat-tailed as an exponential random variable, then $L$ is bounded.  To be more precise, suppose that $f$ is continuously differentiable and for simplicity assume that the distribution of $Z_{i,j}$ is symmetric and $\abs{f'(x)}/(1 -f(x))$ is monotonic for $x$ sufficiently large.  If $\Pr{\abs{Z_{i,j}} \geq t} \geq C \exp(- c t)$ for all $t \geq 0$, then one can show that $L$ is finite.  This property is also essentially equivalent to the requirement that a distribution have bounded {\em hazard rate}. As noted above, this property holds for the logistic distribution, but also for many other common distributions, including the Laplacian, Student's $t$, Cauchy, and others.

\subsection{Room for improvement?}
\label{ssec:optimality}

We now discuss the extent to which Theorems~\ref{thm:recover M} and~\ref{thm:recover distribution} are optimal.  We give three theorems, all proved using information theoretic methods, which show that these results are nearly tight, even when some of our assumptions are relaxed.  Theorem \ref{thm:one_bit_lowerbound} gives a lower bound to nearly match the upper bound on the error in recovering $\M$ derived in Theorem \ref{thm:recover M}.  Theorem \ref{thm:lowerbound} compares our upper bounds to those available without discretization and shows that very little is lost when discretizing to a single bit. Finally, Theorem \ref{thm:distribution lower bound} gives a lower bound matching, up to a constant factor, the upper bound on the error in recovering the distribution $f(\M)$ given in Theorem \ref{thm:recover distribution}.  Theorem \ref{thm:distribution lower bound} also shows that Theorem \ref{thm:recover distribution} does not suffer by dropping the canonical ``spikiness" constraint.

Our lower bounds require a few assumptions, so before we delve into the bounds themselves, we briefly argue that these assumptions are rather innocuous.  First, without loss of generality (since we can always adjust $f$ to account for rescaling $\M$), we assume that $\alpha \ge 1$.  Next, we require that the parameters be sufficiently large so that
\begin{equation}
\label{eq:alphareq}
    \alpha^2 r  \max\{d_1,d_2\}  \geq  C_0
\end{equation}
for an absolute constant $C_0$.  Note that we could replace this with a simpler, but still mild, condition that $d_1 > C_0$.  Finally, we also require that $r \geq c$ where $c$ is either 1 or 4 and that $r \leq O( \min\{d_1,d_2\}/\alpha^2 )$, where $O(\cdot)$ hides parameters (which may differ in each Theorem) that we make explicit below.  This last assumption simply means that we are in the situation where $r$ is significantly smaller than $d_1$ and $d_2$, i.e., the matrix is of approximately low rank.

In the following, let
\begin{equation}
\label{eq:defofK}
    K = \left\{ \M\ : \ \nucnorm{\M} \leq \alpha \sqrt{r d_1d_2}, \|\M\|_\infty \leq \alpha \right\}
\end{equation}
denote the set of matrices whose recovery is guaranteed by Theorem \ref{thm:recover M}.

\subsubsection{Recovery from 1-bit measurements}

\begin{theorem}\label{thm:one_bit_lowerbound}
Fix $\alpha, r,d_1,$ and $d_2$ to be such that $r \ge 4$ and~\eqref{eq:alphareq} holds. Let $\beta_{\alpha}$ be defined as in \eqref{eq:lipschitz}, and suppose that $f'(x)$ is decreasing for $x > 0$. Let $\Omega$ be any subset of $[d_1]\times [d_2]$ with cardinality $\meas $, and let $\Y$ be as in \eqref{eq:observations}. Consider any algorithm which, for any $\M \in K$, takes as input $Y_{i,j}$ for $(i,j) \in \Omega$ and returns $\widehat{\M}$.  Then there exists $\M \in K$ such that with probability at least $3/4$,
\begin{equation}\label{eq:one-bit lower bound}
    \frac{1}{d_1d_2} \| \M - \widehat{\M} \|_{F}^2 \geq \min\left\{C_1, C_2 \alpha\sqrt{\beta_{\frac{3}{4} \alpha}} \sqrt{\frac{r \max\{d_1,d_2\}}{\meas } }\right\}
\end{equation}
as long as the right-hand side of~\eqref{eq:one-bit lower bound} exceeds $r \alpha^2/\min(d_1,d_2)$.  Above, $C_1$ and $C_2$ are absolute constants.\footnote{Here and in the theorems below, the choice of $3/4$ in the probability bound is arbitrary, and can be adjusted at the cost of changing $C_0$ in~\eqref{eq:alphareq} and $C_1$ and $C_2$.  Similarly, $\beta_{\frac{3}{4} \alpha}$ can be replaced by $\beta_{(1 - \epsilon) \alpha}$ for any $\epsilon > 0$.}
\end{theorem}
The requirement that the right-hand side of~\eqref{eq:one-bit lower bound} be larger than $r \alpha^2/\min(d_1,d_2)$ is satisfied as long as $r \leq O( \min\{d_1,d_2\}/\alpha^2 )$.  In particular, it is satisfied whenever
\[
    r \leq C_3 \frac{\min(1, \beta_0) \cdot \min(d_1, d_2)}{\alpha^2}
\]
for a fixed constant $C_3$. Note also that in the latent variable model in \eqref{eq:noisy 1 bit}, $f'(x)$ is simply the probability density of $Z_{i,j}$.  Thus, the requirement that $f'(x)$ be decreasing is simply asking the probability density to have decreasing tails.  One can easily check that this is satisfied for the logistic and probit models.

Note that if $\alpha$ is bounded by a constant and $f$ is fixed (in which case $\beta_\alpha$ and $\beta_{\alpha'}$ are bounded by a constant), then the lower bound of Theorem \ref{thm:one_bit_lowerbound} matches the upper bound given in~\eqref{eq:tight bound on M} up to a constant.  When $\alpha$ is not treated as a constant, the bounds differ by a factor of $\sqrt{\beta_\alpha}$.  In the logistic model $\beta_\alpha \approx e^\alpha$ and so this amounts to the difference between $e^{\alpha/2}$ and $e^{\alpha}$.  The probit model has a similar change in the constant of the exponent.

\subsubsection{Recovery from unquantized measurements}

Next we show that, surprisingly, very little is lost by discretizing to a single bit. In Theorem \ref{thm:lowerbound}, we consider an ``unquantized" version of the latent variable model in \eqref{eq:noisy 1 bit} with Gaussian noise.  That is, let $\Z$ be a matrix of i.i.d.\ Gaussian random variables, and suppose the noisy entries $M_{i,j} + Z_{i,j}$ are observed directly, without discretization. In this setting, we give a lower bound that still nearly matches the upper bound given in Theorem \ref{thm:recover M}, up to the $\beta_\alpha$ term.

\begin{theorem}
\label{thm:lowerbound}
Fix $\alpha, r,d_1,$ and $d_2$ to be such that $r \ge 1$ and~\eqref{eq:alphareq} holds. Let $\Omega$ be any subset of $[d_1]\times [d_2]$ with cardinality $\meas $, and let $\Z$ be a $d_1\times d_2$ matrix with i.i.d.\ Gaussian entries with variance $\sigma^2$. Consider any algorithm which, for any $\M \in K$, takes as input $Y_{i,j} = M_{i,j} + Z_{i,j}$ for $(i,j) \in \Omega$ and returns $\widehat{\M}$.  Then there exists $\M \in K$ such that with probability at least $3/4$,
\begin{equation}
\label{eq:infinite bit lower bound}
    \frac{1}{d_1d_2} \| \M - \widehat{\M} \|_{F}^2 \geq \min\left\{ C_1,  C_2 \alpha \sigma  \sqrt{\frac{r \max\{d_1, d_2\}}{\meas }}\right\}
\end{equation}
as long as the right-hand side of~\eqref{eq:infinite bit lower bound} exceeds $r \alpha^2/\min(d_1, d_2)$.  Above, $C_1$ and $C_2$ are absolute constants.
\end{theorem}
The requirement that the right-hand side of~\eqref{eq:infinite bit lower bound} be larger than $r \alpha^2/\min(d_1, d_2)$ is satisfied whenever
\[r
    \leq C_3 \frac{\min(1, \sigma^2) \cdot \min(d_1, d_2)}{\alpha^2}
\]
for a fixed constant $C_3$.

Following Remark~\ref{specific models}, the lower bound given in \eqref{eq:infinite bit lower bound} matches the upper bound proven in Theorem \ref{thm:recover M} for the solution to~\eqref{eq:max likelihood} up to a constant, as long as $\alpha/\sigma$ is bounded by a constant.  In other words:
\begin{quote}
\emph{When the signal-to-noise ratio is constant, almost nothing is lost by quantizing to a single bit.}
\end{quote}
Perhaps it is not particularly surprising that 1-bit quantization induces little loss of information in the regime where the noise is comparable to the underlying quantity we wish to estimate---however, what {\em is} somewhat of a surprise is that the simple convex program in~\eqref{eq:max likelihood} can successfully recover all of the information contained in these 1-bit measurements.

Before proceeding, we also briefly note that our Theorem~\ref{thm:lowerbound} is somewhat similar to Theorem 3 in~\cite{negahban2010restricted}.  The authors in~\cite{negahban2010restricted} consider slightly different sets $K$: these sets are more restrictive in the sense that it is required that $\alpha \geq \sqrt{32 \log n}$ and less restrictive because the nuclear-norm constraint may be replaced by a general Schatten-p norm constraint.  It was important for us to allow $\alpha = O(1)$ in order to compare with our upper bounds due to the exponential dependence of $\beta_\alpha$ on $\alpha$ in Theorem \ref{thm:recover M} for the probit model.  This led to some new challenges in the proof. Finally, it is also noteworthy that our statements hold for arbitrary sets $\Omega$, while the argument in~\cite{negahban2010restricted} is only valid for a random choice of $\Omega$.

\subsubsection{Recovery of the distribution from 1-bit measurements}

To conclude we address the optimality of Theorem \ref{thm:recover distribution}.  We show that under mild conditions on $f$, any algorithm that recovers the distribution $f(\M)$ must yield an estimate whose Hellinger distance deviates from the true distribution by an amount proportional to $\alpha \sqrt{ r d_1d_2 / \meas }$,  matching the upper bound of~\eqref{eq:simple distribution error bound} up to a constant.  Notice that the lower bound holds even if the algorithm is promised that $\|\M\|_\infty \leq \alpha$, which the upper bound did not require.

\begin{theorem}
\label{thm:distribution lower bound}
Fix $\alpha, r,d_1,$ and $d_2$ to be such that $r \ge 4$ and~\eqref{eq:alphareq} holds. Let $L_1$ be defined as in~\eqref{eq:lipschitz}, and suppose that $f'(x) \geq c$ and $c' \leq f(x) \leq 1 - c'$ for $x \in [-1,1]$, for some constants $c, c' > 0$. Let $\Omega$ be any subset of $[d_1]\times [d_2]$ with cardinality $\meas $, and let $\Y$ be as in \eqref{eq:observations}. Consider any algorithm which, for any $\M \in K$, takes as input $Y_{i,j}$ for $(i,j) \in \Omega$ and returns $\widehat{\M}$.  Then there exists $\M \in K$ such that with probability at least $3/4$,
\begin{equation}\label{eq:distribution lower bound}
	d_H^2(f(\M), f(\Mhat)) \geq \min\left\{ C_1, C_2 \frac{\alpha}{L_1} \sqrt{\frac{r \max\{d_1,d_2\}}{\meas } }\right\}
\end{equation}
as long as the right-hand side of~\eqref{eq:distribution lower bound} exceeds $r \alpha^2/\min(d_1, d_2)$.  Above, $C_1$ and $C_2$ are constants that depend on $c,c'$.
\end{theorem}
The requirement that the right-hand side of~\eqref{eq:distribution lower bound} be larger than $r \alpha^2/\min(d_1, d_2)$ is satisfied whenever
\[
    r \leq C_3 \frac{\min( 1, L_1^{-2}) \cdot \min(d_1, d_2)}{\alpha^2}
\]
for a constant $C_3$ that depends only on $c, c'$.  Note also that the condition that $f$ and $f'$ be well-behaved in the interval $[-1,1]$ is satisfied for the logistic model with $c=1/4$ and $c' = \frac{1}{1 + e} \leq 0.269$.  Similarly, we may take $c = 0.242$ and $c' = 0.159$ in the probit model.

\section{Simulations}
\label{sec:sims}

\subsection{Implementation}

Before presenting the proofs of our main results, we provide algorithms and a suite
of numerical experiments to demonstrate their usefulness in practice.\footnote{The code for these algorithms,
as well as for the subsequent experiments, is available
online at {\color{red}\texttt{http://users.ece.gatech.edu/{\tiny\raisebox{1pt}{${\sim}$}}mdavenport/}}.}
We present algorithms to solve the convex programs \eqref{eq:max likelihood with inf norm} and
\eqref{eq:max likelihood}, and using these we can recover $\M$ (or $f(\M)$) via 1-bit
matrix completion.

\subsubsection{Optimization algorithm}

We begin with the observation that both \eqref{eq:max likelihood with inf norm} and \eqref{eq:max
  likelihood} are instances of the more general formulation
\begin{equation}\label{eq:general form}
\minimize{x} \quad f(x)\quad \st\quad x \in \mathcal{C},
\end{equation}
where $f(x)$ is a smooth convex function from
$\mathbb{R}^d\to\mathbb{R}$, and $\mathcal{C}$ is a closed convex set
in $\mathbb{R}^d$. In particular, defining $\mathcal{V}$ to be the
bijective linear mapping that vectorizes $\mathbb{R}^{d_1\times d_2}$ to
$\mathbb{R}^{d_1d_2}$, we have $f(x) :=
-F_{\Omega,\Y}(\mathcal{V}^{-1}x)$ and, depending on whether we want
to solve \eqref{eq:max likelihood with inf norm} or \eqref{eq:max
  likelihood}, $\mathcal{C}$ equal to either
$\mathcal{V}(\mathcal{C}_1)$ or $\mathcal{V}(\mathcal{C}_2)$, where
\begin{equation}\label{eq:feasible sets}
\mathcal{C}_1 := \{ \X : \norm{\X}_* \leq \tau\}\quad\mathrm{and}\quad
\mathcal{C}_2 := \{ \X : \norm{\X}_* \leq \tau, \norm{\X}_\infty \leq \kappa\}.
\end{equation}

One possible solver for problems of the form \eqref{eq:general form}
is the nonmonotone spectral projected-gradient (SPG) algorithm
proposed by Birgin et al.~\cite{BirginMartinezRaydan2000}. Another
possibility is to use an accelerated proximal-gradient methods for the
minimization of composite functions
\cite{BeckTeboulle2009a,Nesterov2004}, which are useful for solving optimization problems of the form
\[
\minimize{x} \quad f(x) + g(x),
\]
where $f(x)$ and $g(x)$ are convex functions with $g(x)$ possibly
non-smooth. This formulation reduces to \eqref{eq:general form} when
choosing $g(x)$ to be the extended-real indicator function
corresponding to $\mathcal{C}$:
\[
g(x) = \begin{cases} 0 & x \in \mathcal{C}\\ +\infty & \text{otherwise}.\end{cases}
\]
Both algorithms are iterative and require at each iteration the
evaluation of $f(x)$, its gradient $\nabla f(x)$, and an orthogonal
projection onto $\mathcal{C}$ (i.e., the prox-function of $g(x)$)
\begin{equation}\label{eq:orthogonal projection}
\mathcal{P}_{\mathcal{C}}(v) := \argmin_{x}\quad \twonorm{x - v} \quad
\st \quad x
  \in \mathcal{C}.
\end{equation}
For our experiments we use the SPG algorithm, which we describe in
more detail below. The implementation of the algorithm is based on the
SPGL1 code \cite{Software:SPGL1,VandenBergFriedlander2008}.

\subsubsection{Spectral projected-gradient method}

In basic gradient-descent algorithms for unconstrained minimization of
a convex function $f(x)$, iterates are of the form $x_{k+1} = x_{k} -
\alpha_k\nabla f(x_{k})$, where the step length $\alpha_k \in (0,1]$
is chosen such that sufficient descent in the objective function
$f(x)$ is achieved. When the constraint $x\in\mathcal{C}$ is added,
the basic scheme can no longer guarantee feasibility of the
iterates. Projected gradient methods resolve this problem by including
on orthogonal projections back onto the feasible set
\eqref{eq:orthogonal projection} at each iteration.

The nonmonotone SPG algorithm described
in~\cite{BirginMartinezRaydan2000} modifies the basic projected
gradient method in two major ways. First, it scales the initial search
direction using the spectral step-length $\gamma_k$ as proposed by
Barzilai and Borwein \cite{BarzilaiBorwein1988}. Second, it relaxes
monotonicity of the objective values by requiring sufficient descent
relative to the maximum objective over the last $t$ iterates (or $k$
when $k < t$). Two types of line search are considered. The first type
is curvilinear and traces the following path:
\begin{equation}\label{eq:curvilinear line search}
x(\alpha) := \mathcal{P}_{\mathcal{C}}(x_k - \alpha\gamma_k \nabla f(x_k)).
\end{equation}
The second type first determines a projected gradient step, and uses
this to obtain the search direction $d_k$:
\[
d_k = \mathcal{P}_{\mathcal{C}}(x_k - \gamma_k\nabla f(x_k)) - x_k.
\]
Next, a line search is done along the linear trajectory
\begin{equation}\label{eq:linear line search}
x(\alpha) := x_k + \alpha d_k.
\end{equation}
In either case, once the step length $\alpha$ is chosen, we set
$x_{k+1} = x(\alpha)$, and proceed with the next iteration.

In the 1-bit matrix completion formulation proposed in this paper, the
projection onto $\mathcal{C}$ forms the main computational
bottleneck. As a result, it is crucial to keep the number of
projections to a minimum, and our implementation therefore relies
primarily on the line search along the linear trajectory given
by~\eqref{eq:linear line search}. The more expensive curvilinear line
search is used only when the linear one fails. We have observed that
this situation tends to arise only when $x_k$ is near optimal.



The optimality condition for \eqref{eq:general form} is that
\begin{equation}\label{eq:optimality condition}
\mathcal{P}_{\mathcal{C}}(x - \nabla f(x)) = x.
\end{equation}
Our implementation checks if \eqref{eq:optimality condition} is
approximately satisfied. In addition it imposes bounds on the total
number of iterations and the run time.

\subsubsection{Orthogonal projections onto the feasible sets}

It is well known that the orthogonal projection onto the
nuclear-norm ball $\mathcal{C}_1$ amounts to singular-value soft
thresholding \cite{CaiCandesShen2010}. In particular, let $\X = \U
\bSigma \V^*$ with $\bSigma =
\mathrm{diag}(\sigma_1,\ldots,\sigma_d)$, then
\[
\mathcal{P}_{\mathcal{C}}(\X) = \mathcal{S}_{\lambda}(\X) := \U \max\{\bSigma - \lambda \I,0\} \V^*,
\]
where the maximum is taken entrywise, and $\lambda \geq 0$ is the smallest
value for which $\sum_{i=1}^d \max\{\sigma_i - \lambda,0\} \leq \tau$.

Unfortunately, no closed form solution is known for the orthogonal
projection onto $\mathcal{C}_2$. However, the underlying problem
\begin{equation}\label{eq:Projection C2}
\mathcal{P}_{\mathcal{C}_2}(\X) := \argmin_{\Z}\quad \frac{1}{2}\norm{\X -
  \Z}_F^2\quad \st\quad \norm{\Z}_* \leq \tau,\ \norm{\Z}_{\infty} \leq
\kappa,
\end{equation}
can be solved using iterative methods. In particular, we can rewrite
\eqref{eq:Projection C2} as
\begin{equation}\label{eq:ADMM C2}
\minimize{\Z, \W}\quad \half\norm{\X - \W}_F^2\quad \st\quad
  \norm{\W}_{\infty}\leq \kappa,\ \ \norm{\Z}_* \leq \tau,\ \ \W = \Z.
\end{equation}
and apply the alternating-direction method of multipliers (ADMM)
\cite{GabayMercier1976,GlowinskiMarrocco1975}. The augmented
Lagrangian for \eqref{eq:ADMM C2} with respect to the constraint $\W =
\Z$ is given by
\[
\mathcal{L}_\mu(\Y,\W,\Z) = \half\norm{\X-\W}_F^2 -
\langle \Y, \W-\Z\rangle +
{\textstyle\frac{\mu}{2}}\norm{\W-\Z}_F^2 +
\infind{\norm{\W}_{\infty} \leq \kappa} + \infind{\norm{\Z}_* \leq \tau}
\]
The ADMM iterates the following steps to solve \eqref{eq:ADMM C2}:
\begin{description}
\item[Step 0.] Initialize $k = 0$, and select $\mu_k$, $\Y_{k}$,
  $\W_k$, $\Z_k$ such that $\norm{\W_k}_\infty \leq \kappa$ and
  $\norm{\Z_k}_* \le \tau$.
\item[Step 1.] Minimize $\mathcal{L}_{\mu}(\Y_k,W,\Z_k)$ with respect to
  $\W$, which can be rewritten as
\[
\W^{k+1} := \argmin_{\W} \norm{\W - (\X + \Y_k + \mu \Z_k) /
  (1+\mu))}_F^2\quad\st\quad\norm{\W}_\infty \leq \kappa.
\]
This is exactly the orthogonal projection of $\B = (\X + \Y_k + \mu
\Z_k) / (1+\mu)$ onto $\{\W \mid \norm{\W}_\infty \leq \kappa\}$, and
gives $\W_{k+1}(i,j) = \min\{\kappa,\max\{-\kappa,\B(i,j)\}\}$.
\item[Step 2.] Minimize $\mathcal{L}_{\mu}(\Y_k, \W_{k+1},\Z)$ with
respect to $\Z$. This gives
\[
\Z_{k+1} = \argmin_{\Z} \norm{\Z - (\W_{k+1} - 1/\mu_k\Y_{k})}_F^2\quad
\st\quad \norm{\Z}_* \leq \tau,
\]
and simplifies to $\Z_{k+1} = \mathcal{P}_{\mathcal{C}_1}(\W_{k+1} -
1/\mu_k\Y_{k})$.
\item[Step 3.] Update $\Y_{k+1} = \Y_{k} - \mu(\W_{k+1} - \Z_{k+1})$,
 set $\mu_{k+1} = 1.05\,\mu_{k}$, and increment $k$.
\item[Step 4.] Return $\Z = \Z_{k}$ when $\norm{\W_k - \Z_k}_F \leq
  \varepsilon$ and $\norm{\Z_k}_{\infty} - \kappa \leq
  \varepsilon$ for some sufficiently small $\varepsilon >
  0$. Otherwise, repeat steps 1--4.
\end{description}


\subsection{Synthetic experiments}

To evaluate the performance of this algorithm in practice and to
confirm the theoretical results described above, we first performed a
number of synthetic experiments.  In particular, we constructed a
random $d\times d$ matrix $\M$ with rank $r$ by forming $\M = \M_1
\M_2^*$ where $\M_1$ and $\M_2$ are $d \times r$ matrices with entries
drawn i.i.d.\ from a uniform distribution on $[-\frac12,\frac12]$.
The matrix is then scaled so that $\|\M\|_\infty = 1$.  We then
obtained 1-bit observations by adding Gaussian noise of variance
$\sigma^2$ and recording the sign of the resulting value.

We begin by comparing the performance of the algorithms
in~\eqref{eq:max likelihood with inf norm} and~\eqref{eq:max
  likelihood} over a range of different values of $\sigma$.  In this
experiment we set $d = 500$, $r=1$, and $\meas  = 0.15 d^2$, and we
measured performance of each approach using the squared Frobenius norm
of the error (normalized by the norm of the original matrix $\M$) and
averaged the results over 15 draws of $\M$.\footnote{In evaluating
  these algorithms we found that it was beneficial in practice to
  follow the recovery by a ``debiasing'' step where the recovered
  matrix $\Mhat$ is forced to be rank $r$ by computing the SVD of
  $\Mhat$ and hard thresholding the singular values.  In cases where
  we report the Frobenius norm of the error, we performed this
  debiasing step, although it does not dramatically impact the
  performance.} The results are shown in Figure~\ref{fig:sigmaExp}.
We observe that for both approaches, the performance is poor when
there is too little noise (when $\sigma $ is small) and when there is
too much noise (when $\sigma$ is large). These two regimes correspond
to the cases where the noise is either so small that the observations
are essentially noise-free or when the noise is so large that each
observation is essentially a coin toss.  In the regime where the noise
is of moderate power, we observe better performance for both
approaches.  Perhaps somewhat surprisingly, we find that for much of
this range, the approach in~\eqref{eq:max likelihood} appears to
perform almost as well as~\eqref{eq:max likelihood with inf norm},
even though we do not have any theoretical guarantees
for~\eqref{eq:max likelihood}.  This suggests that adding the
infinity-norm constraint as in~\eqref{eq:max likelihood with inf norm}
may have only limited practical benefit, despite the key role this
constraint played in our analysis. By using the simpler program
in~\eqref{eq:max likelihood} one can greatly simplify the projection
step in the algorithm, so in practice this approach may be preferable.

\begin{figure}
\centering
\includegraphics{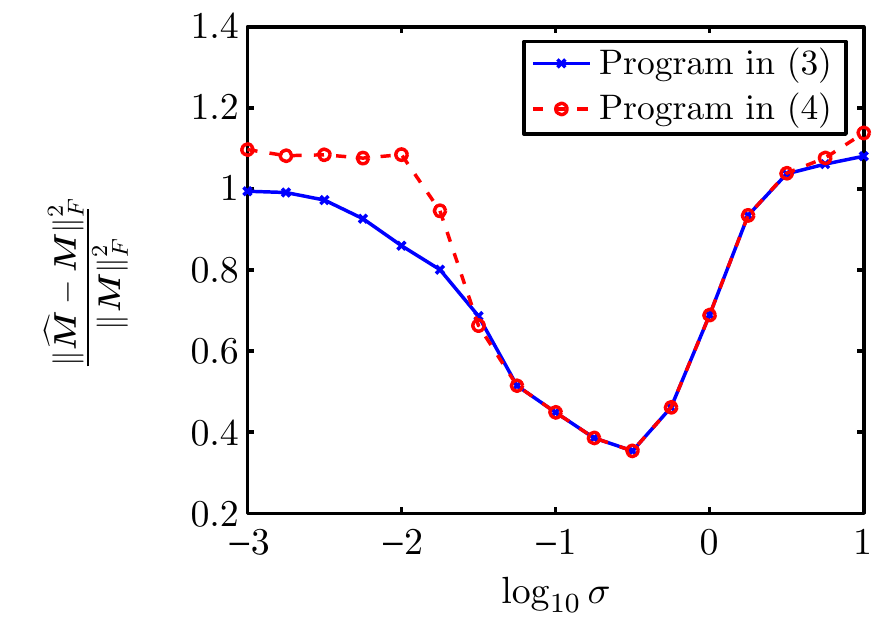}
\caption{\small \sl Average relative error in the results obtained using algorithms~\eqref{eq:max likelihood with inf norm} and~\eqref{eq:max likelihood} over a range of different values of $\sigma$.  We observe that for both approaches, the performance is poor when there is too little noise (when $\sigma $ is small) and when there is too much noise (when $\sigma$ is large). In the case of moderate noise, we observe that~\eqref{eq:max likelihood} appears to perform almost as well as~\eqref{eq:max likelihood with inf norm}, even though we do not have any theoretical guarantees for~\eqref{eq:max likelihood}. \label{fig:sigmaExp}}
\end{figure}

We also conducted experiments evaluating the performance of
both~\eqref{eq:max likelihood with inf norm} and~\eqref{eq:max
  likelihood} as a function of $\meas $ for a particular choice of
$\sigma$.  The results showing the impact of $\meas $ on~\eqref{eq:max
  likelihood} are shown in Figure~\ref{fig:mExp} (the results
for~\eqref{eq:max likelihood with inf norm} at this noise level are
almost indistinguishable).  In this experiment we set $d = 200$, and
chose $\sigma \approx 0.18$ such that $\log_{10}(\sigma) = 0.75$,
which lies in the regime where the noise is neither negligible nor
overwhelming. We considered matrices with rank $r = 3, 5, 10$ and
evaluated the performance over a range of
$\meas $. Figure~\ref{fig:mExp}(a) shows the performance in terms of the
relative Frobenius norm of the error, and Figure~\ref{fig:mExp}(b)
shows the performance in terms of the Hellinger distance between the
recovered distributions.  Consistent with our theoretical results, we
observe a decay in the error (under both performance metrics) that
appears to behave roughly on the order of $\meas ^{-1/2}$.

\begin{figure}
\centering
\begin{tabular}{c c}
\hspace{-.2in} \includegraphics{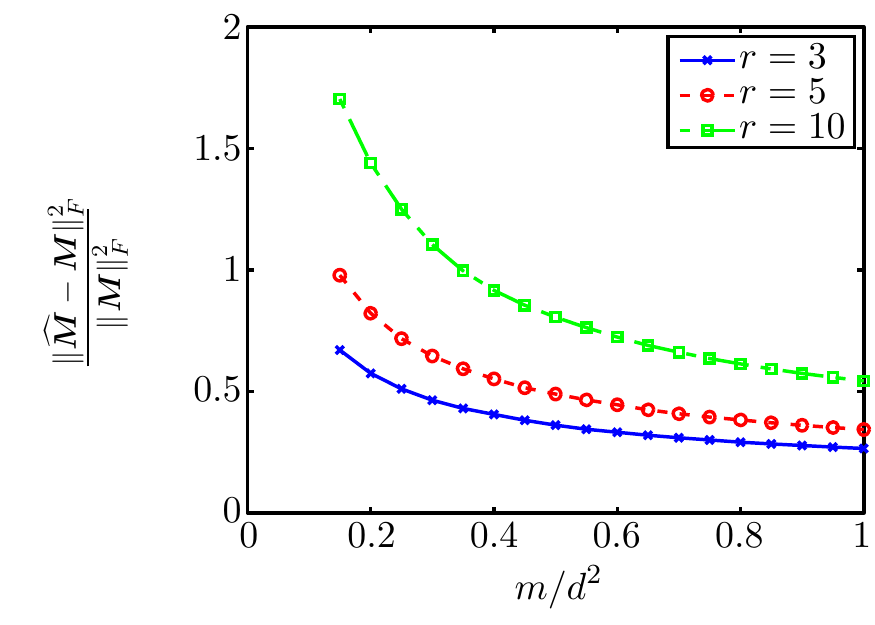} & \hspace{-.2in} \includegraphics{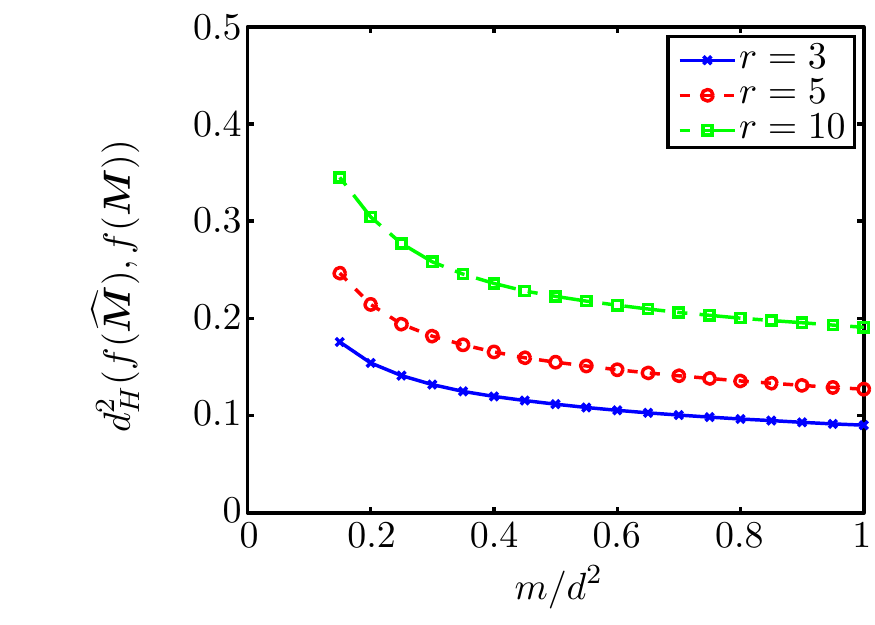} \\
{\small \sl (a)} & {\small \sl (b)}
\end{tabular}
\caption{\small \sl The impact of $\meas $ on~\eqref{eq:max likelihood}. (a) shows the performance in terms of the relative Frobenius norm of the error. (b) shows the performance in terms of the Hellinger distance between the recovered distributions.  In both cases, we observe a decay in the error (under both performance metrics) that appears to behave roughly on the order of $\meas ^{-1/2}$.   \label{fig:mExp}}
\end{figure}

\subsection{Collaborative filtering}

To evaluate the performance of our algorithm in a practical setting,
we consider the MovieLens (100k) data set, which is available for
download at~\url{http://www.grouplens.org/node/73}.  This data set
consists of 100,000 movie ratings from 1000 users on 1700 movies, with
each rating occurring on a scale from 1 to 5.  For testing purposes,
we converted these ratings to binary observations by comparing each
rating to the average rating for the entire dataset (which is
approximately 3.5).  We then apply the algorithm in~\eqref{eq:max
  likelihood} (using the logistic model of $f(x) = \frac{e^x}{1+e^x}$)
on a subset of 95,000 ratings to recover an estimate of $\M$.
However, since our only source of data is the quantized ratings, there
is no ``ground truth'' against which to measure the accuracy of our
recovery.  Thus, we instead evaluate our performance by checking to
see if the estimate of $\M$ accurately predicts the sign of the
remaining 5000 unobserved ratings in our dataset.  The result of this
simulation is shown in the first line of Table~\ref{fig:movielens}, which gives the accuracy in predicting
whether the unobserved ratings are above or below the average rating
of 3.5.

By comparison, the second line in the table shows the results obtained
using a ``standard'' method that uses the raw ratings (on a scale from
1 to 5) and tries to minimize the nuclear norm of the recovered matrix
subject to a constraint that requires the Frobenius norm of the
difference between the recovered and observed entries to be
sufficiently small.  We implement this traditional approach using the
TFOCS software package~\cite{tfocs}, and evaluate the results using
the same error criterion as the 1-bit matrix completion
approach---namely, we compare the recovered ratings to the average
(recovered) rating.  This approach depends on a number of input
parameters: $\alpha$ in~\eqref{eq:max likelihood with inf norm}, the
constraint on the Frobenius norm in the traditional case, as well as
the internal parameter $\mu$ in TFOCS. We determine the parameter
values by simply performing a grid search and selecting those
values that lead to the best performance.

Perhaps somewhat surprisingly, the 1-bit approach performs significantly better than the traditional one, {\em even though the traditional approach is given more information} in the form of the raw ratings, instead of the binary observations.\footnote{While not reported in the table, we also observed that the 1-bit approach is relatively insensitive to the choice of $\alpha$, so that this improvement in performance does not rely on a careful parameter setting.} The intuition as to how this might be possible is that the standard approach is likely paying a significant penalty for actually requiring that the recovered matrix yields numerical ratings close to ``1'' or ``5'' when a user's true preference could extend beyond this scale.

\begin{table}
\centering
\begin{tabular}{|r||c | c | c | c | c || c|}
\hline
Original rating & 1 & 2 & 3 & 4 & 5 & Overall \\ \hline \hline
1-bit matrix completion & 79\% & 73\% & 58\% & 75\% & 89\% & 73\% \\ \hline
``Standard'' matrix completion & 64\% & 56\% & 44\% & 65\% & 75\% & 60\% \\ \hline
\end{tabular}
\caption{\small \sl Results of a comparison between a 1-bit matrix completion approach and ``standard'' matrix completion on the MovieLens 100k dataset, as evaluated by predicting whether the unobserved ratings were above or below average.  Overall, the 1-bit approach achieves an accuracy of 73\%, which is significantly better than the 60\% accuracy achieved by standard methods.  The 1-bit approach also outperforms the standard approach for each possible individual rating, although we see that both methods perform poorly when the true rating is very close to the average rating. \label{fig:movielens}}
\end{table}

\section{Discussion}
\label{sec:disc}

Many of the applications of matrix completion consider discrete observations, often in the form of binary measurements. However, matrix completion from  noiseless binary measurements is extremely ill-posed, even if one collects a binary measurement for each of the matrix entries. Fortunately, when there are some  stochastic variations (noise) in the observations, recovery becomes well-posed. In this paper we have shown  that the unknown matrix can be accurately and efficiently recovered from binary measurements in this setting.  When the infinity norm of the unknown matrix is bounded by a constant, our error bounds are tight to within a constant and even match what is possible for undiscretized data.  We have also shown that the binary probability distribution can be reconstructed over the entire matrix without any assumption on the infinity-norm, and we have provided a matching lower bound (up to a constant).

Our theory considers approximately low-rank matrices---in particular, we assume that the singular values belong to a scaled Schatten-1 ball.  It would be interesting to see whether more accurate reconstruction could be achieved under the assumption that the unknown matrix has precisely $r$ nonzero singular values.  We conjecture that the Lagrangian formulation of the problem could be fruitful for this endeavor.  It would also be interesting to study whether our ideas can be extended to deal with measurements that are quantized to more than 2 (but still a small number) of different values, but we leave such investigations for future work.

\section*{Acknowledgements}
We would like to thank Roman Vershynin for helpful discussions and Wenxin Zhou for pointing out an error in an earlier version of the proof of Theorem \ref{thm:recover distribution with constraints}.

\appendix
\section{Proofs of the main results}
\label{sec:proofs}

We now provide the proofs of the main theorems presented in Section~\ref{sec:main}.  To begin, we first define some additional notation that we will need for the proofs. For two probability distributions $\mathcal{P}$ and $\mathcal{Q}$ on a finite set $A$, $D(\mathcal{P}\|\mathcal{Q})$ will denote the Kullback-Leibler (KL) divergence,
\[
    D(\mathcal{P}\| \mathcal{Q}) = \sum_{x \in A} \mathcal{P}(x) \log\left( \frac{ \mathcal{P}(x) }{\mathcal{Q}(x)} \right),
\]
where $\mathcal{P}(x)$ denotes the probability of the outcome $x$ under the distribution $\mathcal{P}$.  We will abuse this notation slightly by overloading it in two ways.  First, for scalar inputs $p,q \in [0,1]$, we will set
\[
    D(p\|q) = p \log\left( \frac{p}{q} \right) + (1-p) \log\left( \frac{1-p}{1-q} \right).
\]
Second, for two matrices $\mP, \mQ \in [0,1]^{d_1 \times d_2}$, we define
\[
    D(\mP \| \mQ) = \frac{1}{d_1 d_2} \sum_{i,j} D(P_{i,j} \| Q_{i,j}).
\]
We first prove Theorem \ref{thm:recover distribution}.  Theorem \ref{thm:recover M} will then follow from an approximation argument.  Finally, our lower bounds will be proved in Section \ref{sec:lowerbounds} using information theoretic arguments.

\subsection{Proof of Theorem~\ref{thm:recover distribution}}
\label{ssec:recover_distribution}

We will actually prove a slightly more general statement, which will be helpful in the proof of Theorem \ref{thm:recover M}.
We will assume that $\inftynorm{\M} \leq \gamma$, and we will modify the program~\eqref{eq:max likelihood} to enforce $\inftynorm{\X} \leq \gamma$.  That is, we will consider the program
\begin{equation}
\label{eq:special max likelihood}
    \Mhat = \argmax_{\X}  \loglike_{\Omega, \Y}(\X)  \qquad \mathrm{subject~to} \qquad \nucnorm{\X} \leq \alpha \sqrt{r d_1 d_2} \quad \mathrm{and} \quad \inftynorm{\X} \leq \gamma.
\end{equation}
We will then send $\gamma \to \infty$ to recover the statement of Theorem~\ref{thm:recover distribution}. Formally, we prove the following theorem.

\begin{theorem}
\label{thm:recover distribution with constraints}
Assume that $\nucnorm{\M} \leq \alpha \sqrt{r d_1 d_2}$ and $\inftynorm{\M} \leq \gamma$.  Suppose that $\Omega$ is chosen at random following the binomial model of Section \ref{sec:model} and satisfying $\E|\Omega| = \meas $.  Suppose that $\Y$ is generated as in~\eqref{eq:observations}, and let $L_\gamma$ be as in~\eqref{eq:lipschitz}. Let $\Mhat$ be the solution to~\eqref{eq:special max likelihood}.  Then with probability at least $1 - C_1/(d_1+d_2)$,
\begin{equation}
\label{eq:distribution error bound 2}
    d_H^2(f(\Mhat), f(\M) ) \leq C_2 L_\gamma \alpha \sqrt{\frac{r(d_1 + d_2)}{\meas }}\sqrt{1 + \frac{(d_1 + d_2) \log(d_1 d_2)}{\meas }}.
\end{equation}
Above, $C_1$ and $C_2$ are absolute constants.
\end{theorem}

For the proof of Theorem \ref{thm:recover distribution with constraints}, it will be convenient to work with the function
\begin{equation*}
\bar{\loglike}_{\Omega, \Y}(\X) = \loglike_{\Omega, \Y}(\X) - \loglike_{\Omega, \Y}(\mathbf{0})
\end{equation*}
rather than with $\loglike_{\Omega, \Y}$ itself.
The key will be to establish the following concentration inequality.
\begin{lemma}
\label{lem:concentration}
Let $G \subset \R^{d_1 \times d_2}$ be
\[
    G = \left\{ \X \in \R^{d_1 \times d_2} \ : \ \nucnorm{\X} \leq \alpha \sqrt{r d_1d_2} \right\}
\]
for some $r \leq \min \{d_1,d_2\}$ and $\alpha \geq 0$.  Then
\begin{equation}
\label{eq:supProbBound}
    \Pr{\sup_{\X \in G}  | \bar{\loglike}_{\Omega, \Y}(\X) - \E{ \bar{\loglike}_{\Omega, \Y}(\X) }| \geq C_0 \alpha L_\gamma \sqrt{r}\sqrt{\meas (d_1 + d_2) + d_1 d_2 \log(d_1 d_2)}} \le \frac{C_1}{d_1+d_2},
\end{equation}
where $C_0$ and $C_1$ are absolute constants and the probability and the expectation are both over the choice of $\Omega$ and the draw of $Y$.
\end{lemma}

We will prove this lemma below, but first we show how it implies Theorem~\ref{thm:recover distribution with constraints}. To begin, notice that for any choice of $\X$,
\begin{align*}
\Exp{ \bar{\loglike}_{\Omega, \Y}(\X) - \bar{\loglike}_{\Omega, \Y}(\M)} &=  \Exp{ \loglike_{\Omega, \Y}(\X) - \loglike_{\Omega, \Y}(\M) }\\
&=\frac{\meas }{d_1d_2} \sum_{i,j} \left( f(M_{i,j}) \log\left( \frac{ f(X_{i,j})}{f(M_{i,j})}\right) + (1 - f(M_{i,j})) \log\left( \frac{1 - f(X_{i,j})}{1 - f(M_{i,j})}\right)\right) \\
 & =  -\meas D(f(\M) \| f(\X) ),
\end{align*}
where the expectation is over both $\Omega$ and $\Y$. Next, note that by assumption $\M \in G$.
Then, we have for any $\X \in G$
\begin{align*}
\bar{\loglike}_{\Omega, \Y}(\X) - \bar{\loglike}_{\Omega,\Y}(\M) &=
\Exp{ \bar{\loglike}_{\Omega, \Y}(\X) - \bar{\loglike}_{\Omega,\Y}(\M) } + \left( \bar{\loglike}_{\Omega,\Y}(\X) - \Exp{ \bar{\loglike}_{\Omega,\Y}(\X) } \right) - \left( \bar{\loglike}_{\Omega,\Y}(\M) - \Exp{ \bar{\loglike}_{\Omega,\Y}(\M) }\right) \\
&\leq \Exp{ \bar{\loglike}_{\Omega, \Y}(\X) - \bar{\loglike}_{\Omega,\Y}(\M) } + 2\sup_{\X \in G} \left| \bar{\loglike}_{\Omega,\Y}(\X) - \Exp{ \bar{\loglike}_{\Omega,\Y}(\X) } \right| \\
&= -\meas D(f(\M) \| f(\X) ) + 2\sup_{\X \in G} \left| \bar{\loglike}_{\Omega,\Y}(\X) - \Exp{ \bar{\loglike}_{\Omega,\Y}(\X) } \right|.
\end{align*}
Moreover, from the definition of $\Mhat$ we also have that $\Mhat \in G$ and $\loglike_{\Omega, \Y}( \Mhat) \ge \loglike_{\Omega, \Y}( \M)$.  Thus
\begin{align*}
0 &\leq -\meas D(f(\M) \| f(\Mhat) ) + 2\sup_{\X \in G} \left| \bar{\loglike}_{\Omega,\Y}(\X) - \Exp{ \bar{\loglike}_{\Omega,\Y}(\X) } \right|.
\end{align*}
Applying Lemma~\ref{lem:concentration}, we obtain that with probability at least $1-C_1/(d_1+d_2)$,  we have
\[
    0 \leq -\meas D( f(\M) \| f(\Mhat)) + 2C_0 \alpha L_\gamma \sqrt{r}\sqrt{\meas (d_1 + d_2) + d_1 d_2 \log(d_1 d_2)}
\]
In this case, by rearranging and applying the fact that $\sqrt{d_1 d_2} \le d_1 + d_2$, we obtain
\begin{equation}
\label{eq:KLbound}
    D( f(\M) \| f(\Mhat)) \leq 2 C_0 \alpha L_\gamma \sqrt{\frac{r(d_1 + d_2)}{\meas }}\sqrt{1 + \frac{(d_1 + d_2) \log(d_1 d_2)}{\meas }}
\end{equation}
Finally, we note that the KL divergence can easily be bounded below by the Hellinger distance:
\[
    d_H^2(p,q) \leq D(p\|q).
\]
This is a simple consequence of Jensen's inequality combined with the fact that $1-x \le - \log x$. Thus, from~\eqref{eq:KLbound} we obtain
\[
    d^2_H( f(\M) , f(\Mhat) ) \leq 2 C_0 \alpha L_\gamma \sqrt{\frac{r(d_1 + d_2)}{\meas }}\sqrt{1 + \frac{(d_1 + d_2) \log(d_1 d_2)}{\meas }},
\]
which establishes Theorem~\ref{thm:recover distribution with constraints}.  Theorem \ref{thm:recover distribution} then follows by taking the limit as $\gamma \to \infty$.

\begin{remark}[Relationship to existing theory]
\label{remark:generror}
The crux of the proof of the theorem 
involves bounding the quantity
\begin{equation}\label{eq:classerror}
\bar{\loglike}_{\Omega, \Y}( \Mhat) - \bar{\loglike}_{\Omega, \Y}(\M)
\end{equation}
via Lemma \ref{lem:concentration}.
In the language of statistical learning theory~\cite{boucheron2005theory, kar2013generalization}, \eqref{eq:classerror} is a generalization error, with respect to a \em loss function \em corresponding to $\bar{\loglike}$. In that literature, there are standard approaches to proving bounds on quantities of this type. We use a few of these tools (namely, symmetrization and contraction principles) in the proof of Lemma~\ref{lem:concentration}.  However, our proof also deviates from these approaches in a few ways.  Most notably, we use a moment argument to establish concentration, while in the classification literature it is common to use approaches based on the method of bounded differences.  In our problem setup, using bounded differences fails to yield a satisfactory answer because in the most general case we do not impose an $\ell_\infty$-norm constraint on $\M$.  This in turn implies that our differences are not well-bounded. 
Specifically, 
consider the empirical process whose moment we bound in~\eqref{eq:supbound simple}, i.e.,
\[\sup_{\X \in G} \ip{ \mE \had \Delta_{\Omega} } {\X} .\]
The supremum is a function of the independent random variables $E_{i,j} \cdot \Delta_{i,j}$.  To check the effectiveness of the method of bounded differences, suppose $E_{i,j} \Delta_{i,j} = 1$ (for some $i,j$), and is replaced by $-1$.  Then the empirical process can change by as much as  $2 \max_{X \in G} \inftynorm{X} = \alpha\sqrt{r d_1 d_2}$, which is too large to yield an effective answer.

The fact that $\bar{\loglike}$ is not well-bounded in this way is also why the generalization error bounds of~\cite{srebro2004generalization, srebro2004maximum} do not immediately generalize to provide results analogous to Theorem~\ref{thm:recover distribution}.  To obtain this result, we must choose a loss function such that~\eqref{eq:classerror} reduces to the KL divergence, and unfortunately, this loss function is not well-behaved.

Finally, we also note that symmetrization and contraction principles, applied to bound empirical processes, are key tools in the theory of unquantized matrix completion.  The empirical process that we must bound to prove Lemma \ref{lem:concentration} is a discrete analog of a similar process considered in Section 5 of \cite{negahban2010restricted}.
\end{remark}

\begin{proof}[Proof of Lemma \ref{lem:concentration}]

We begin by noting that for any $h > 0$, by using Markov's inequality we have that
\begin{align}
& \Pr{\sup_{\X \in G}  | \bar{\loglike}_{\Omega, \Y}(\X) - \E{ \bar{\loglike}_{\Omega, \Y}(\X) }| \geq C_0 \alpha L_\gamma \sqrt{r}\sqrt{\meas (d_1 + d_2) + d_1 d_2 \log(d_1 d_2)}} \notag \\
& = \Pr{\sup_{\X \in G}  | \bar{\loglike}_{\Omega, \Y}(\X) - \E{ \bar{\loglike}_{\Omega, \Y}(\X) }|^h \geq \left(C_0 \alpha L_\gamma \sqrt{r}\sqrt{\meas (d_1 + d_2) + d_1 d_2 \log(d_1 d_2)}\right)^h } \notag \\
& \le \frac{\Exp{\sup_{\X \in G}  | \bar{\loglike}_{\Omega, \Y}(\X) - \E{ \bar{\loglike}_{\Omega, \Y}(\X) }|^h}}{\left(C_0 \alpha L_\gamma \sqrt{r}\sqrt{\meas (d_1 + d_2) + d_1 d_2 \log(d_1 d_2)}\right)^h}. \label{eq:Markov}
\end{align}
The bound in~\eqref{eq:supProbBound} will follow by combining this with an upper bound on $\Exp{\sup_{\X \in G} |\bar{\loglike}_{\Omega, \Y}(\X) - \E{\bar{\loglike}_{\Omega, \Y}(\X)}|^h}$ and setting $h = \log(d_1 + d_2)$.  Towards this end, note that we can write the definition of $\bar{\loglike}_{\Omega, \Y}$ as
\[
    \bar{\loglike}_{\Omega, \Y}(\X) = \sum_{i,j} \left( \ind{(i,j) \in \Omega} \left( \ind{Y_{i,j} = 1} \log\left( \frac{f(X_{i,j}) }{f(0)} \right) + \ind{Y_{i,j} = -1} \log\left( \frac{ 1 - f(X_{i,j}) }{1 - f(0)} \right) \right)\right).
\]
By a symmetrization argument (Lemma 6.3 in~\cite{LT}),
\begin{align*}
& \Exp{\sup_{\X \in G} |\bar{\loglike}_{\Omega, \Y}(\X) - \E{\bar{\loglike}_{\Omega, \Y}(\X)} |^h }\\
& \qquad \qquad \leq 2^h \, \Exp{ \sup_{\X \in G} \left| \sum_{i,j} \eps_{i,j} \ind{(i,j) \in \Omega} \left( \ind{Y_{i,j} = 1} \log\left( \frac{f(X_{i,j})}{ f(0)}  \right) + \ind{Y_{i,j} = -1} \log\left( \frac{1 - f(X_{i,j})}{ 1- f(0)} \right) \right)\right|^h},
\end{align*}
where the $\eps_{i,j}$ are i.i.d.\ Rademacher random variables and the expectation in the upper bound is with respect to both $\Omega$ and $\Y$ as well as with respect to the $\eps_{i,j}$.  To bound the latter term, we apply a contraction principle (Theorem 4.12 in~\cite{LT}). By the definition of $L_\gamma$ and the assumption that $\|\Mhat\|_{\infty} \leq \gamma$, both
\[ \frac{1}{L_\gamma} \log\left( \frac{f(x)}{f(0)} \right) \qquad \text{and} \qquad  \frac{1}{L_\gamma} \log\left(\frac{1 - f(x)}{1 - f(0)}\right)\]
 are contractions that vanish at $0$.  Thus, up to a factor of $2$, the expected value of the supremum can only decrease when these are replaced by $X_{i,j}$ and $-X_{i,j}$ respectively.  We obtain
\begin{align}
\Exp{\sup_{\X \in G} \left|\bar{\loglike}_{\Omega, \Y}(\X)  - \E{\bar{\loglike}_{\Omega, \Y}(\X)} \right|^h } &\leq 2^h (2L_\gamma)^h \ \Exp{ \sup_{\X \in G} \left|\sum_{i,j} \eps_{i,j} \ind{(i,j) \in \Omega} \left( \ind{Y_{i,j} = 1} X_{i,j} - \ind{Y_{i,j} = -1} X_{i,j} \right)\right|^h} \notag \\
& = (4L_\gamma)^h \ \Exp{ \sup_{\X \in G} \left|\ip{ \Delta_{\Omega} \had \mE \had \Y }{ \X }\right|^h }, \label{eq:supbound1}
\end{align}
where $\mE$ denotes the matrix with entries given by $\eps_{i,j}$, $\Delta_\Omega$ denotes the indicator matrix for $\Omega$ (so that $[\Delta_\Omega]_{i,j} = 1$ if $(i,j)\in \Omega$ and 0 otherwise), and~$\had$ denotes the Hadamard product.  Using the facts that the distribution of $\mE \had \Y$ is the same as the distribution of $\mE$ and that $| \langle \A, \B \rangle | \le \| \A \| \| \B \|_{*}$, we have that
\begin{align}
\Exp{ \sup_{\X \in G} \left|\ip{ \Delta_{\Omega} \had \mE \had  \Y }{ \X }\right|^h } & = \Exp{ \sup_{\X \in G} \left|\ip{ \mE \had \Delta_{\Omega} } {\X} \right|^h} \label{eq:supbound simple} \\
& \leq  \Exp{ \sup_{\X \in G} \norm{\mE \had \Delta_{\Omega} }^h \nucnorm{\X}^h} \notag \\
& = \left(\alpha \sqrt{d_1 d_2 r}\right)^h \ \Exp{ \|\mE \had \Delta_{\Omega} \|^h}, \label{eq:supbound2}
\end{align}
To bound $\Exp{\|\mE \had \Delta_{\Omega}\|^h}$, observe that $\mE \had \Delta_{\Omega}$ is a matrix with i.i.d.\ zero mean entries and thus by Theorem~1.1 of~\cite{seginer2000expected},
\[
    \Exp{\opnorm{\mE \had \Delta_{\Omega}}^h} \leq  C\left(\Exp{\max_{1 \leq i \leq d_1} \left(\sum_{j = 1}^{d_2} \Delta_{i,j}\right)^{\frac{h}{2}}} + \Exp{\max_{1 \leq j \leq d_2} \left(\sum_{i = 1}^{d_1} \Delta_{i,j}\right)^{\frac{h}{2}}}\right)
\]
for some constant $C$.  This in turn implies that
\begin{equation}
\label{eq:opnormbound1}
    \left( \Exp{\opnorm{\mE \had \Delta_{\Omega}}^h} \right)^{\frac{1}{h}} \leq  C^{\frac{1}{h}} \left( \left(\Exp{\max_{1 \leq i \leq d_1} \left(\sum_{j = 1}^{d_2} \Delta_{i,j}\right)^{\frac{h}{2}}} \right)^{\frac{1}{h}} +   \left(\Exp{\max_{1 \leq j \leq d_2} \left(\sum_{i = 1}^{d_1} \Delta_{i,j}\right)^{\frac{h}{2}}} \right)^{\frac{1}{h}} \right).
\end{equation}
We first focus on the row sum $\sum_{j=1}^{d_2} \Delta_{i,j}$ for a particular choice of $i$. Using Bernstein's inequality, for all $t > 0$ we have
\[
\Pr{ \left| \sum_{j = 1}^{d_2} \left( \Delta_{i,j} - \frac{\meas }{d_1d_2}\right) \right| > t } \leq 2\exp\left( \frac{-t^2/2}{\meas /d_1 + t/3}\right).
\]
In particular, if we set $t \geq 6 \meas /d_1$, then for each $i$ we have
\begin{equation}
\label{eq:tailcompare}
    \Pr{ \left| \sum_{j=1}^{d_2} \left( \Delta_{i,j} - \frac{\meas }{d_1d_2}\right)\right| > t} \leq 2\exp(-t) = 2\Pr{ W_i > t },
\end{equation}
where $W_1, \ldots, W_{d_1}$ are i.i.d.\ exponential random variables.

Below we use the fact that for any positive random variable $q$ we can write $\E q = \int_0^\infty \Pr{q \geq t}$, allowing us to bound
{\allowdisplaybreaks
\begin{align*}
\left(\Exp{\max_{1 \leq i \leq d_1} \left(\sum_{j = 1}^{d_2} \Delta_{i,j}\right)^{\frac{h}{2}}}\right)^{\frac{1}{h}} &\leq \sqrt{\frac{\meas }{d_1}} + \left( \Exp{\max_{1 \leq i \leq d_1} \left| \sum_{j=1}^{d_2} \left( \Delta_{i,j} - \frac{\meas }{d_1d_2}  \right) \right|^{\frac{h}{2}} }\right)^{\frac{1}{h}} \\
&\leq \sqrt{\frac{\meas }{d_1}} + \left( \Exp{\max_{1 \leq i \leq d_1} \left| \sum_{j=1}^{d_2} \left( \Delta_{i,j} - \frac{\meas }{d_1d_2}  \right) \right|^{h} }\right)^{\frac{1}{2h}} \\
&= \sqrt{\frac{\meas }{d_1}} + \left(\int_0^\infty \Pr{ \max_{1 \leq i \leq d_1} \left| \sum_{j=1}^{d_2} \left( \Delta_{i,j} - \frac{\meas }{d_1d_2} \right)\right|^h \geq t}\,dt  \right)^{\frac{1}{2h}}\\
&\leq \sqrt{\frac{\meas }{d_1}} + \left( \left(\frac{6 \meas }{d_1}\right)^h + \int_{(6 \meas /d_1)^h}^\infty \Pr{ \max_{1 \leq i \leq d_1} \left| \sum_{j=1}^{d_2} \left( \Delta_{i,j} - \frac{\meas }{d_1d_2} \right)\right|^h \geq t}\,dt  \right)^{\frac{1}{2h}}\\
&\leq \sqrt{\frac{\meas }{d_1}} + \left( \left(\frac{6 \meas }{d_1}\right)^h + 2 \int_{(6 \meas /d_1)^h}^\infty \Pr{ \max_{1 \leq i \leq d_1} W_i^h \geq t}\,dt  \right)^{\frac{1}{2h}}\\
&\leq \sqrt{\frac{\meas }{d_1}} + \left( \left(\frac{6 \meas }{d_1}\right)^h + 2 \Exp{ (\max_{1 \leq i \leq d_1} W_i)^h }  \right)^{\frac{1}{2h}}.
\end{align*}
Above, we have used the triangle inequality in the first line, followed by Jensen's inequality in the second line.  In the fifth line, \eqref{eq:tailcompare}, along with independence, allows us to introduce $\max_i W_i$.  By standard computations for exponential random variables,
\begin{align*}
\Exp{ \max_{1 \leq i \leq d_1} W_i^h } &\leq \Exp{ \left| \max_{1 \leq i \leq d_1} W_i - \log d_1 \right|^h } + \log^h(d_1) \\
&\leq 2 h! + \log^h(d_1) .
\end{align*}
Thus, we have
\begin{align*}
\left(\Exp{\max_{1 \leq i \leq d_1} \left(\sum_{j = 1}^{d_2} \Delta_{i,j}\right)^{\frac{h}{2}}}\right)^{\frac{1}{h}}
&\leq \sqrt{\frac{\meas }{d_1}} + \left( \left(\frac{6 \meas }{d_1}\right)^h + 2\left(\log^h(d_1) + 2(h!)\right)   \right)^\frac{1}{2h}\\
&\leq (1 + \sqrt{6})\sqrt{ \frac{\meas }{d_1}} + 2^\frac{1}{2h}\left(\sqrt{\log d_1} + 2^\frac{1}{2h}\sqrt{h}\right)\\
&\leq (1 + \sqrt{6})\sqrt{ \frac{\meas }{d_1}} + (2 + \sqrt{2}) \sqrt{ \log(d_1 + d_2) },
\end{align*}
using the choice $h = \log(d_1 + d_2) \geq 1$ in the final line.}

A similar argument bounds the column sums, and thus from~\eqref{eq:opnormbound1} we conclude that
\begin{align*}
\left(\Exp{\opnorm{\mE \had \Delta_{\Omega}}^{h}}\right)^{\frac{1}{h}} & \leq C^{\frac{1}{h}} \left( (1 + \sqrt{6}) \left( \sqrt{ \frac{\meas }{d_1} } + \sqrt{ \frac{\meas }{d_2} }\right) + (2 + \sqrt{2}) \sqrt{ \log(d_1 + d_2)} \right) \\
 & \leq C^{\frac{1}{h}} \left( (1 + \sqrt{6}) \left( \sqrt{ \frac{2 \meas (d_1+d_2)}{d_1 d_2} } \right) + (2 + \sqrt{2}) \sqrt{ \log(d_1 + d_2)} \right) \\
 & \leq C^{\frac{1}{h}} 2 (1+ \sqrt{6}) \sqrt{ \frac{ \meas (d_1 + d_2) + d_1 d_2 \log(d_1 + d_2) }{d_1 d_2} },
\end{align*}
where the second and third inequalities both follow from Jensen's inequality.  Combining this with~\eqref{eq:supbound1} and~\eqref{eq:supbound2}, we obtain
\[
    \left(\Exp{\sup_{\X \in G} \left|\bar{\loglike}_{\Omega, \Y}(\X)  - \E{\bar{\loglike}_{\Omega, \Y}(\X)} \right|^h } \right)^{\frac{1}{h}} \leq C^{\frac{1}{h}} 8 (1 + \sqrt{6})  \alpha L_\gamma \sqrt{r} \sqrt{ \meas (d_1 + d_2) + d_1 d_2 \log(d_1 + d_2)}.
\]
Plugging this into~\eqref{eq:Markov} we obtain that the probability in~\eqref{eq:Markov} is upper bounded by
\[
    C \left( \frac{8 (1 + \sqrt{6})}{C_0}  \right)^{\log(d_1 + d_2)} \le \frac{C}{d_1 + d_2},
\]
provided that $C_0 \ge 8(1+\sqrt{6})/e$, which establishes the lemma.
\end{proof}

\subsection{Proof of Theorem \ref{thm:recover M}}

The proof of Theorem~\ref{thm:recover M} follows immediately from Theorem \ref{thm:recover distribution with constraints} (with $\gamma = \alpha$) combined with the following lemma.

\begin{lemma}
Let $f$ be a differentiable function and let $\inftynorm{\M}, \|\Mhat\|_{\infty} \leq \alpha$.  Then
\[
    d_H^2(f(\M), f(\Mhat)) \geq \inf_{\abs{\xi} \leq \alpha} \frac{(f'(\xi))^2}{8 f(\xi)(1 - f(\xi))} \frac{\| \M - \Mhat\|_{F}^2}{d_1 d_2}.
\]
\end{lemma}
\begin{proof}
For any pair of entries $x = M_{i,j}$ and $y = \Mhatel_{i,j}$,  write
\[
    \left(\sqrt{f(x)} - \sqrt{f(y)}\right)^2 + \left(\sqrt{ 1- f(x)} - \sqrt{1 - f(y)}\right)^2 \geq \frac{1}{2} \left( \left( \sqrt{f(x)} - \sqrt{f(y)} \right) - \left( \sqrt{f(x)} - \sqrt{f(y)} \right)\right)^2.
\]
Using Taylor's theorem to expand the quantity inside the square, for some $\xi$ between $x$ and $y$,
\begin{align*}
\left(\sqrt{f(x)} - \sqrt{f(y)}\right)^2 + \left(\sqrt{ 1- f(x)} - \sqrt{1 - f(y)}\right)^2
&\geq \frac{1}{2} \left( \frac{ f'(\xi)(y-x) }{ 2 \sqrt{f(\xi)}} + \frac{ f'(\xi)(y-x)}{2 \sqrt{ 1- f(\xi)}} \right)^2\\
&\geq \frac{1}{8} (f'(\xi))^2 (y-x)^2 \left( \frac{1}{f(\xi)} + \frac{1}{1 - f(\xi)}\right)\\
&= \frac{(f'(\xi))^2}{8 f(\xi)(1 - f(\xi))} (y-x)^2.
\end{align*}
The lemma follows by summing across all entries and dividing by $d_1 d_2$.
\end{proof}

\subsection{Lower bounds}
\label{sec:lowerbounds}

The proofs of our lower bounds each follow a similar outline, using classical information theoretic techniques that have also proven useful in the context of compressed sensing~\cite{candes2011well,raskutti2009minimax}.  At a high level, our argument involves first showing the existence of a set of matrices $\mathcal{X}$, so that for each $\X^{(i)} \neq \X^{(j)} \in \mathcal{X}$, $\|\X^{(i)} - \X^{(j)}\|_{F}$ is large.  We will imagine obtaining measurements of a randomly chosen matrix in $\mathcal{X}$ and then running an arbitrary recovery procedure.  If the recovered matrix is sufficiently close to the original matrix, then we could determine which element of $\mathcal{X}$ was chosen.  However, Fano's inequality will imply that the probability of correctly identifying the chosen matrix is small, which will induce a lower bound on how close the recovered matrix can be to the original matrix.

In the proofs of Theorems \ref{thm:one_bit_lowerbound}, \ref{thm:lowerbound}, and \ref{thm:distribution lower bound},  we will assume without loss of generality that $d_2 \ge d_1$.  Before providing these proofs, however, we first consider the construction of the set $\mathcal{X}$.

\subsubsection{Packing set construction}

\begin{lemma}
\label{lem:existsX}
Let $K$ be defined as in~\eqref{eq:defofK}, let $\gamma \leq 1$ be such that $\frac{r}{\gamma^2}$ is an integer, and suppose that $\frac{r}{\gamma^2} \leq d_1$. There is a set $\mathcal{X} \subset K$ with
\[
    |\mathcal{X}| \ge \exp\left( \frac{r d_2}{ 16 \gamma^2}\right)
\]
with the following properties:
\begin{enumerate}
\item For all $\X \in \mathcal{X}$, each entry has $|X_{i,j}| = \alpha \gamma$.
\item \label{item:distance} For all $\X^{(i)}, \X^{(j)} \in \mathcal{X}$, $i \neq j$,
\[
    \|\X^{(i)} - \X^{(j)}\|_{F}^2 >  \frac{\alpha^2 \gamma^2 d_1d_2}{2}.
\]
\end{enumerate}
\end{lemma}
\begin{proof}
We use a probabilistic argument.  The set $\mathcal{X}$ will by constructed by drawing
\begin{equation}
\label{eq:size X}
    |\mathcal{X}| = \left\lceil \exp\left( \frac{r d_2}{ 16 \gamma^2}\right) \right\rceil
\end{equation}
matrices $\X$ independently from the following distribution.  Set $B = \frac{r}{\gamma^2}$.  The matrix will consist of blocks of dimensions $B  \times d_2$, stacked on top of each other.  The entries of the first block (that is, $X_{i,j}$ for $(i,j) \in [B] \times [d_2]$) will be i.i.d. symmetric random variables with values $\pm \alpha \gamma$.  Then $\X$ will be filled out by copying this block as many times as will fit.  That is,
\[
    X_{i, j} := X_{i',j} \text{\ \ where \ \ } i' = i \,\,(\text{mod} \, B) + 1.
\]

Now we argue that with nonzero probability, this set will have all the desired properties.  For $\X \in \mathcal{X}$,
\[
    \|\X\|_\infty = \alpha \gamma \leq \alpha.
\]
Further, because $\rank{\X} \le B$,
\[
    \nucnorm{\X} \leq \sqrt{B} \fronorm{\X} = \sqrt{\frac{r}{\gamma^2}}\sqrt{d_1d_2}\alpha\gamma = \alpha \sqrt{r d_1d_2}.
\]
Thus $\mathcal{X} \subset K$, and all that remains is to show that $\mathcal{X}$ satisfies requirement \ref{item:distance}.

For $\X, \W$ drawn from the above distribution,
\begin{align*}
\fronorm{\X - \W}^2 &= \sum_{i,j}(X_{i,j} - W_{i,j})^2\\
&\geq  \left\lfloor \frac{d_1}{B}\right\rfloor \sum_{i \in [B]} \sum_{j \in [d_2]} (X_{i,j} - W_{i,j})^2\\
&= 4\alpha^2\gamma^2  \left\lfloor \frac{d_1}{B}\right\rfloor \sum_{i \in [B ]} \sum_{j \in [d_2]} \delta_{i,j} \\
&=: 4\alpha^2\gamma^2 \left\lfloor \frac{d_1}{B}\right\rfloor Z(\X,\W).
\end{align*}
where the $\delta_{i,j}$ are independent $0/1$ Bernoulli random variables with mean $1/2$.  By Hoeffding's inequality and a union bound,
\[
    \Pr{ \min_{\X \neq \W \in \mathcal{X}}  Z(\X,\W)  \leq \frac{ d_2 B}{4} } \leq \binom{ \abs{\mathcal{X}}}{2} \exp(-B d_2/8).
\]
One can check that for $\X$ of the size given in~\eqref{eq:size X}, the right-hand side of the above tail bound is less than 1, and thus the event that $Z(\X, \W) > d_2 B/4$ for all $\X \neq \W \in \mathcal{X}$ has non-zero probability.  In this event,
\[
    \fronorm{\X - \W}^2 > \alpha^2 \gamma^2 \left\lfloor \frac{d_1}{B} \right\rfloor d_2 B \geq  \frac{\alpha^2 \gamma^2 d_1d_2}{2},
\]
where the second inequality uses the assumption that $d_1 \geq B$ and the fact that $\lfloor x \rfloor \geq x/2$ for all $x \geq 1$.   Hence, requirement (\ref{item:distance}) holds with nonzero probability and thus the desired set exists.
\end{proof}

\subsubsection{Proof of Theorem \ref{thm:one_bit_lowerbound}}

Before we prove Theorem \ref{thm:one_bit_lowerbound}, we will need the following lemma about the KL divergence.
\begin{lemma}\label{lem:kl_div}
Suppose that $ x , y \in (0,1)$.  Then
\[
    D(x\|y) \leq \frac{ (x - y)^2 }{ y (1 - y) }.
\]
\end{lemma}
\begin{proof}
Without loss of generality, we may assume that $x \le  y$.  Indeed, $D( 1 - x \| 1 - y ) = D(x\|y)$, and either $x \le  y$ or $1 - x \le 1 - y$.
Let $z = y - x$.  A simple computation shows that
\[
    \frac{\partial}{\partial z}  D(x \| x + z ) = \frac{z}{(x + z)(1 - x - z)}.
\]
Thus, by Taylor's theorem, there is some $\xi \in [0, z]$ so that
\[
    D(x \| y) = D( x\| x) + z \left( \frac{ \xi }{ (x + \xi)(1 - x - \xi) }\right).
\]
Since the right hand side is increasing in $\xi$, we may replace $\xi$ with $z$ and conclude
\[
    D(x \| y) \leq \frac{ (y - x)^2 }{y(1 - y)},
\]
as desired.
\end{proof}

Now, for the proof of Theorem \ref{thm:one_bit_lowerbound}, we choose $\epsilon$ so that
\begin{equation}\label{eq:defeps}
    \epsilon^2 = \min\left\{ \frac{1}{1024}, C_2 \alpha\sqrt{\beta_{3\alpha/4}} \sqrt{\frac{r d_2}{\meas } }\right\},
\end{equation}
where $C_2$ is an absolute constant to be specified later. We will next use Lemma \ref{lem:existsX} to construct a set $\mathcal{X}$, choosing $\gamma$ so that $\frac{r}{\gamma^2}$ is an integer and
\[
	\frac{4\sqrt{2} \epsilon}{\alpha} \leq \gamma \leq \frac{8 \epsilon}{\alpha}.
\]
We can make such a choice because $\epsilon \leq 1/32$ and $r \geq 4$.  We verify that such a choice for $\gamma$ satisfies the requirements of Lemma \ref{lem:existsX}.  Indeed, since $\epsilon \leq 1/32$ and $\alpha \geq 1$ we have $\gamma \leq 1/4 < 1$.  Further, we assume in the theorem that the right-hand side of~\eqref{eq:defeps} is larger than $C r \alpha^2 / d_1$ which implies that $r/\gamma^2 \leq d_1$ for an appropriate choice of $C$.

Let $\mathcal{X}'_{\alpha/2,\gamma}$ be the set whose existence is guaranteed by Lemma~\ref{lem:existsX} with this choice of $\gamma$, and with $\alpha/2$ instead of $\alpha$. We will construct $\mathcal{X}$ by setting
\[
	\mathcal{X} := \left\{ \X' + \alpha \left(1 - \frac{\gamma}{2} \right) \bone \ : \ \X' \in \mathcal{X}'_{\alpha/2,\gamma} \right\}
\]
Note that $\mathcal{X}$ has the same size as $\mathcal{X}'_{\alpha/2,\gamma}$, i.e., $|\mathcal{X}|$ satisfies~\eqref{eq:size X}.  $\mathcal{X}$ also has the same bound on pairwise distances
\begin{equation}\label{eq:spacing}
    \| \X^{(i)} - \X^{(j)} \|_{F}^2 \geq \frac{\alpha^2 \gamma^2 d_1d_2}{8} \geq 4 d_1d_2\epsilon^2,
\end{equation}
and every entry of $\X \in \mathcal{X}$ has
\[
    |X_{i,j}| \in \{  \alpha , \alpha' \},
\]
where $\alpha' = (1 - \gamma)\alpha$. Further, since for $\X' \in \mathcal{X}'_{\alpha/2,\gamma}$,
\[
    \nucnorm{\X' + \alpha(1 - \gamma/2) \bone} \leq \nucnorm{\X'} + \alpha(1 - \gamma/2) \sqrt{d_1d_2} \leq \alpha \sqrt{rd_1d_2}
\]
for $r \ge 4$ as in the theorem statement.

Now suppose for the sake of a contradiction that there exists an algorithm such that for any $\X \in K$, when given access to the measurements $\Y_\Omega$, returns an $\widehat{\X}$ such that
\begin{equation}\label{eq:recover}
    \frac{1}{d_1d_2}\| \X - \widehat{\X}\|^2_F < \epsilon^2
\end{equation}
with probability at least $1/4$.   We will imagine running this algorithm on a matrix $\X$ chosen uniformly at random from $\mathcal{X}$.   Let
\[
    \X^* = \argmin_{\X^{(i)} \in \mathcal{X}} \| \X^{(i)} - \widehat{\X} \|_F^2.
\]
It is easy to check that if~\eqref{eq:recover} holds, then $\X^* = \X$.  Indeed, for any $\X' \in \mathcal{X}$ with $\X' \neq \X$, from~\eqref{eq:recover} and~\eqref{eq:spacing} we have that
\[
    \| \X' - \widehat{\X} \|_F = \| \X' - \X + \X - \widehat{\X} \|_F \ge \| \X' - \X\|_F - \|\X - \widehat{\X} \|_F > 2 \sqrt{d_1 d_2} \epsilon - \sqrt{d_1 d_2} \epsilon = \sqrt{d_1 d_2} \epsilon.
\]
At the same time, since $\X \in \mathcal{X}$ is a candidate for $\X^*$, we have that
\[
    \| \X^* - \widehat{\X} \|_F \le \| \X - \widehat{\X} \|_F \le \sqrt{d_1 d_2} \epsilon.
\]
Thus, if~\eqref{eq:recover} holds, then $\| \X^* - \widehat{\X} \|_F < \| \X' - \widehat{\X} \|_F$ for any $\X' \in \mathcal{X}$ with $\X' \neq \X$, and hence we must have $\X^* = \X$.  By assumption,~\eqref{eq:recover} holds with probability at least $1/4$, and thus
\begin{equation}\label{eq:unlikely_recovery}
\Pr{\X \neq \X^*} \leq \frac34.
\end{equation}
We will show that this probability must in fact be large, generating our contradiction.

By a variant of Fano's inequality
\begin{equation}\label{eq:fanovariant2}
    \Pr{\X \neq \X^*} \geq 1 - \frac{ \max_{\X^{(k)} \neq \X^{(\ell)} } D( \Y_\Omega |\X^{(k)}\ \|\ \Y_\Omega | \X^{(\ell)} )  + 1 }{\log|\mathcal{X}|}.
\end{equation}
Because each entry of $\Y$ is independent,\footnote{Note that here, to be consistent with the literature we are referencing regarding Fano's inequality, $D$ is defined slightly differently than elsewhere in the paper where we would weight $D$ by $1/d_1 d_2$.}
\[
    D := D(\Y_\Omega | \X^{(k)} \ \| \ \Y_\Omega |\X^{(\ell)} ) = \sum_{(i,j) \in \Omega} D( Y_{i,j}| X^{(k)}_{i,j} \ \|\ Y_{i,j} | X^{(\ell)}_{i,j}).
\]
Each term in the sum is either $0$, $D( \alpha \| \alpha' )$, or $D( \alpha' \| \alpha)$.  By Lemma \ref{lem:kl_div}, all of these are bounded above by
\[
    D( Y_{i,j} | X^{(k)}_{i,j} \ \|\ Y_{i,j} | X^{(\ell)}_{i,j} ) \leq \frac{ ( f(\alpha) - f(\alpha'))^2 }{ f(\alpha')(1 - f(\alpha') )},
\]
and so, from the intermediate value theorem, for some $\xi \in [ \alpha', \alpha ]$,
\[
    D \leq \meas \frac{ (f(\alpha) - f(\alpha'))^2}{ f(\alpha') ( 1 -f (\alpha'))} \leq \meas \frac{ (f'(\xi))^2 (\alpha - \alpha')^2 }{ f(\alpha') (1 - f(\alpha')) }.
\]
Using the assumption that $f'(x)$ is decreasing for $x > 0$ and the definition of $\alpha' = (1 - \gamma)\alpha$, we have
\[
    D \leq \frac{ \meas (\gamma\alpha)^2 }{\beta_{\alpha'}} \leq\frac{ 64 \meas \epsilon^2 }{ \beta_{\alpha'} }.
\]
Then from \eqref{eq:fanovariant2} and \eqref{eq:unlikely_recovery},
\begin{equation}
\label{eq:theorem 3 fano}
    \frac{1}{4} \leq 1 - \Pr{\X \neq \X^*} \leq \frac{D + 1}{\log|\mathcal{X}|} \leq 16 \gamma^2 \left(\frac{ \frac{64 \meas \epsilon^2 }{\beta_{\alpha'}} + 1 }{r d_2}\right) \leq 1024 \epsilon^2 \left(\frac{ \frac{64 \meas \epsilon^2 }{\beta_{\alpha'}} + 1 }{\alpha^2 r d_2}\right).
\end{equation}
We now show that for appropriate values of $C_0$ and $C_2$, this leads to a contradiction.  First suppose that $64 \meas \epsilon^2 \le \beta_{\alpha'}$.  In this case we have
\[
    \frac{1}{4} \le 1024 \epsilon^2 \frac{2}{\alpha^2 r d_2},
\]
which together with~\eqref{eq:defeps} implies that $\alpha^2 r d_2 \le 8$.  If we set $C_0 > 8$ in~\eqref{eq:alphareq}, then this would lead to a contradiction.  Thus, suppose now that $64 \meas \epsilon^2 > \beta_{\alpha'}$. Then~\eqref{eq:theorem 3 fano} simplifies to
\[
   \frac{1}{4} < \frac{ 1024 \cdot 128 \cdot \meas \epsilon^4 }{\beta_{\alpha'} \alpha^2 r d_2}.
\]
Thus
\[
    \epsilon^2 > \frac{\alpha \sqrt{\beta_{\alpha'}}}{512\sqrt{2}} \sqrt{\frac{ r d_2}{  \meas }}.
\]
Note $\beta$ is increasing as a function of $\alpha$ and $\alpha' \geq 3\alpha/4$ (since $\gamma \leq 1/4$).  Thus, $\beta_{\alpha'} \geq \beta_{3\alpha/4}$. Setting $C_2 \le 1/512\sqrt{2}$ in~\eqref{eq:defeps} now leads to a contradiction, and hence~\eqref{eq:recover} must fail to hold with probability at least $3/4$, which proves the theorem.

\subsubsection{Proof of Theorem \ref{thm:lowerbound}}
Choose $\epsilon$ so that
\begin{equation}\label{eq:epsdef4}
	 \epsilon^2 = \min\left\{ \frac{1}{16}, C_2 \alpha\sigma \sqrt{ \frac{r d_2}{\meas }} \right\}
\end{equation}
for an absolute constant $C_2$ to be determined later. As in the proof of Theorem \ref{thm:one_bit_lowerbound}, we will consider running such an algorithm on a random element in a set $\mathcal{X} \subset K$. For our set $\mathcal{X}$, we will use the set whose existence is guaranteed by Lemma \ref{lem:existsX}.  We will set $\gamma$ so that $\frac{r}{\gamma^2}$ is an integer and
\[
    \frac{2\sqrt{2} \epsilon}{\alpha} \leq \gamma \leq  \frac{4  \epsilon }{\alpha}.
\]
This is possible since $\epsilon \le 1/4$ and $r, \alpha \geq 1$.  One can check that $\gamma$ satisfies the assumptions of Lemma \ref{lem:existsX}.

Now suppose that $\X \in \mathcal{X}$ is chosen uniformly at random, and let $\Y = \left.(\X + \Z)\right|_\Omega$ as in the statement of the theorem.  Let $\widehat{\X}$ be any estimate of $\X$ obtained from $\Y_\Omega$. We begin by bounding the mutual information $I(\X;\widehat{\X})$ in the following lemma (which is analogous to \cite[Equation 9.16]{cover}).
\begin{lemma}\label{lem:mutualInformation}
\[
    I(\X ; \widehat{\X}) \leq \frac{\meas }{2}\log\left( \sigma^2 + \left( \alpha^2 \gamma^2 \right) \right).
\]
\end{lemma}
\begin{proof}
We begin by noting that
\[
    I(\X_\Omega ; \Y) = h(\X_\Omega + \Z_\Omega) - h(\X_\Omega + \Z_\Omega | \X_\Omega ) = h(\X_\Omega + \Z_\Omega) - h(\Z_\Omega),
\]
where $h$ denotes the differential entropy.  Let $\mathbf{\xi}$ denote a matrix of i.i.d.\ $\pm 1$ entries.  Then
\[
    h(\X_\Omega \had \mathbf{\xi} + \Z_\Omega) = h( (\X_\Omega + \Z_\Omega) \had \mathbf{\xi} ) \geq h( (\X_\Omega + \Z_\Omega) \had \mathbf{\xi}\ \mid\ \mathbf{\xi} ) = h(\X_\Omega + \Z_\Omega),
\]
and so, letting $\widetilde{\X} = \X \had \mathbf{\xi}$,
\[
    I(\X_\Omega ; \Y) \leq h(\widetilde{\X}_\Omega + \Z_\Omega) - h(\Z_\Omega).
\]
Treating $\widetilde{\X}_\Omega + \Z_\Omega$ as a random vector of length $\meas $, we compute the covariance matrix as
\[
    \bSigma := \Exp{ \vecto(\widetilde{\X}_\Omega + \Z_\Omega) \vecto(\widetilde{\X}_\Omega + \Z_\Omega)^T } = \left(\sigma^2 + (\alpha\gamma)^2 \right)\I_{\meas} .
\]
By Theorem 8.6.5 in~\cite{cover},
\[
    h(\widetilde{\X}_\Omega + \Z_\Omega) \leq \frac{1}{2} \log\left((2\pi e)^{\meas } \det(\bSigma)\right) = \frac{1}{2} \log\left( (2\pi e)^{\meas}  (\sigma^2 + (\alpha \gamma)^2)^{\meas} \right).
\]
We have that $h(\Z_\Omega) = \frac{1}{2} \log \left( (2\pi e)^{\meas } \sigma^{2 \meas } \right)$, and so
\[
    I(\X_\Omega; \Y ) \leq \frac{\meas }{2} \log\left( 1 + \left(\frac{\alpha \gamma}{\sigma}\right)^2\right).
\]
Then the data processing inequality implies
\[
    I(\X;\widehat{\X}) \leq \frac{\meas }{2} \log\left( 1 + \left(\frac{\alpha \gamma}{\sigma}\right)^2 \right),
\]
which establishes the lemma.
\end{proof}

We now proceed by using essentially the same argument as in the proof of Theorem~\ref{thm:one_bit_lowerbound}.  Specifically, we suppose for the sake of a contradiction that there exists an algorithm such that for any $\X \in K$, when given access to the measurements $\Y_\Omega$, returns an $\widehat{\X}$ such that
\begin{equation}\label{eq:recover2}
    \frac{1}{d_1d_2}\| \X - \widehat{\X}\|^2_F < \epsilon^2
\end{equation}
with probability at least $1/4$.  As before, if we set
\[
    \X^* = \argmin_{\X^{(i)} \in \mathcal{X}} \| \X^{(i)} - \widehat{\X} \|_F^2
\]
then we can show that if~\eqref{eq:recover2} holds, then $\X^* = \X$.  Thus, if ~\eqref{eq:recover2} holds with probability at least $1/4$ then
\begin{equation}\label{eq:unlikely_recovery2}
    \Pr{\X \neq \X^*} \leq \frac34.
\end{equation}
However, by Fano's inequality, the probability that $\X \neq \widehat{\X}$ is at least
\[
    \Pr{\X \neq \widehat{\X}} \geq \frac{H(\X|\widehat{\X}) - 1 }{\log(|\mathcal{X}|)} = \frac{ H(\X) - I(\X;\widehat{\X}) - 1 }{\log(|\mathcal{X}|)} \geq 1 - \frac{ I(\X;\widehat{\X}) + 1}{\log|\mathcal{X}|}
\]
Plugging in $|\mathcal{X}|$ from Lemma \ref{lem:existsX} and $I(\X;\widehat{\X})$ from Lemma \ref{lem:mutualInformation}, and using the inequality $\log(1 + z) \leq z$, we obtain
\[
    \Pr{\X \neq \widehat{\X}} \geq 1 - \frac{16 \gamma^2 }{r d_2 } \left( \frac{\meas }{2} \left(\frac{\alpha \gamma}{\sigma} \right)^2 + 1 \right).
\]
Combining this with~\eqref{eq:unlikely_recovery2} and using the fact that $\gamma \leq 4 \epsilon / \alpha$, we obtain
\[
    \frac{1}{4} \leq \frac{ 256 \epsilon^2 }{\alpha^2 r d_2 } \left(  8 \meas \left( \frac{\epsilon^2}{\sigma^2} \right)  + 1 \right).
\]
We now argue, as before, that this leads to a contradiction.  Specifically, if $8 \meas \epsilon^2/\sigma^2 \le 1$, then together with~\eqref{eq:epsdef4} this implies that $\alpha^2 r d_2 \le 128$.  If we set $C_0 > 128$ in~\eqref{eq:alphareq}, then this would lead to a contradiction.  Thus, suppose now that $8 \meas \epsilon^2/\sigma^2 > 1$, in which case we have
\[
    \epsilon^2 > \frac{\alpha \sigma}{128} \sqrt{\frac{ r d_2 }{\meas }}.
\]
Thus, setting $C_2 \le 1/128$ in~\eqref{eq:epsdef4} leads to a contradiction, and hence~\eqref{eq:recover2} must fail to hold with probability at least $3/4$, which proves the theorem.

\subsubsection{Proof of Theorem \ref{thm:distribution lower bound}}

The proof of Theorem \ref{thm:distribution lower bound} also mirrors the proof of Theorem \ref{thm:one_bit_lowerbound}.  The main difference is the observation that the set constructed in Lemma \ref{lem:existsX} also works with the Hellinger distance.  We begin as before by choosing $\epsilon$ so that
\begin{equation}\label{eq:defofeps}
    \epsilon^2 = \min\left\{ \frac{c}{16}, C_2 \frac{\alpha}{L_1}  \sqrt{ \frac{ r d_2 }{\meas } }\right\},
\end{equation}
where $C_2$ is an absolute constant to be determined. Set $\gamma$ to be an integer so that
\[
    2\sqrt{2}\frac{\epsilon}{\alpha c} \leq \gamma \leq \frac{4 \epsilon}{\alpha c}.
\]
This is possible since by assumption $\alpha \ge 1$ and $\epsilon \leq \frac{c}{4}$.  One can check that $\gamma$ satisfies the assumptions of Lemma \ref{lem:existsX}.

As in the proof of Theorem \ref{thm:one_bit_lowerbound}, we will consider running such an algorithm on a random element in a set $\mathcal{X} \subset K$. For our set $\mathcal{X}$, we will use the set whose existence is guaranteed by Lemma \ref{lem:existsX}.  Note that since the Hellinger distance is bounded below by the Frobenius norm, we have that for all $\X^{(i)} \neq \X^{(j)} \in \mathcal{X}$,
\[
    d_H^2( f(\X^{(i)} ) - f(\X^{(j)}) ) \geq \| f(\X^{(i)}) - f(\X^{(j)} ) \|_{F}^2 \geq c^2 \| \X^{(i)} - \X^{(j)} \|_{F}^2 > \frac{c^2}{2} \alpha^2 \gamma^2 d_1 d_2  \geq 4 d_1d_2 \epsilon^2.
\]
Now suppose for the sake of a contradiction that there exists an algorithm such that for any $\X \in K$, when given access to the measurements $\Y_\Omega$, returns an $\widehat{\X}$ such that
\begin{equation}\label{eq:recover3}
    d_H^2(f(\X), f(\widehat{\X}) ) < \epsilon^2
\end{equation}
with probability at least $1/4$. If we set
\[
    \X^* = \argmin_{\X^{(i)} \in \mathcal{X}} d_H^2( f(\X^{(i)}) - f(\widehat{\X}) )
\]
then we can show that if~\eqref{eq:recover3} holds, then $\X^* = \X$.  Thus, if ~\eqref{eq:recover3} holds with probability at least $1/4$ then
\begin{equation}\label{eq:unlikely_recovery3}
    \Pr{\X \neq \X^*} \leq \frac34.
\end{equation}
However, we may again apply Fano's inequality as in~\eqref{eq:fanovariant2}. Using Lemma~\ref{lem:kl_div} we have
\[
    D(Y_{i,j}|X_{i,j}^{(k)} \ \|\ Y_{i,j}|X_{i,j}^{(\ell)} ) \leq \frac{ (f(\alpha \gamma) - f(-\alpha\gamma))^2}{ f(\alpha\gamma)(1 - f(\alpha\gamma))} \leq \frac{ 4(f'(\xi))^2 \alpha^2 \gamma^2 }{f(\alpha \gamma)(1 - f(\alpha \gamma))} \leq \frac{ 4f^2(\xi) L^2_{\alpha \gamma} \alpha^2 \gamma^2 }{ f(\alpha \gamma)(1 - f(\alpha \gamma))},
\]
for some $|\xi| \leq \alpha \gamma$, where $L_{\alpha \gamma}$ is as in~\eqref{eq:lipschitz}.  By the assumption that $c' < |f(x)| < 1-c'$ for $|x| < 1$, and that
\[
    \alpha \gamma \leq \alpha \left( \frac{4 \epsilon}{\alpha c} \right) \leq \frac{4\epsilon}{c} \leq 1,
\]
we obtain
\[
    D(Y_{i,j}|X_{i,j}^{(k)} \ \|\ Y_{i,j}|X_{i,j}^{(\ell)} ) \leq \frac{ 4c' L_1^2 \alpha^2 \gamma^2 }{1 - c'} \leq C' L_1^2 \epsilon^2,
\]
where $C' = 64 c'/(c^2(1-c'))$. Thus, from~\eqref{eq:fanovariant2}, we have
\[
    \frac{1}{4} \leq \frac{ C' \meas  L_1^2 \epsilon^2 + 1}{\log|\mathcal{X}| } \leq \frac{256}{c^2} \epsilon^2 \left( \frac{ C' \meas  L_1^2 \epsilon^2 + 1 }{\alpha^2 r d_2} \right).
\]
We now argue once again that this leads to a contradiction.  Specifically, if $C' \meas L_1^2 \epsilon^2 \le 1$, then together with~\eqref{eq:defofeps} this implies that $\alpha^2 r d_2 \le 128/c$.  If we set $C_0 > 128/c$ in~\eqref{eq:alphareq}, then this would lead to a contradiction.  Thus, suppose now that $C' \meas L_1^2 \epsilon^2  > 1$, in which case we have
\[
    \epsilon^2 > \frac{c}{32 \sqrt{2C'}} \frac{\alpha}{L_1} \sqrt{ \frac{ r d_2 }{\meas } }.
\]
Thus setting $C_2 \le c/32 \sqrt{2C'}$ in~\eqref{eq:defofeps} leads to a contradiction, and hence~\eqref{eq:recover3} must fail to hold with probability at least $3/4$, which proves the theorem.

\bibliographystyle{abbrv}
\bibliography{mainbib}

\end{document}